\documentclass[11pt,a4]{amsart}

\usepackage{parskip}
\usepackage{amsfonts,amsmath,amssymb,bbm}
\usepackage{verbatim}
 \usepackage{setspace}
\usepackage{color}
\usepackage{hyperref}

\usepackage{enumerate}
\usepackage{pdfsync}
\usepackage{color}
\numberwithin{equation}{section}
\newtheorem{theorem}{Theorem}[section]
\newtheorem{proposition}[theorem]{Proposition}
\newtheorem{lemma}[theorem]{Lemma}

\theoremstyle{definition}

\theoremstyle{remark}
\newtheorem{remark}[theorem]{Remark}
\usepackage{color}

\newcommand{\spnu}[2]{\langlerangle{#1}_{#2}}
\newcommand{\e}{{\bf{e}}}

\newcommand{\eab}{{\bf{e}}_{\alpha,\beta}}
\newcommand{\lnua}{{\rm{L}}^{2}(\eab)}
\newcommand{\eb}{{\bf{e}}_{\beta}}
\newcommand{\lnub}{{\rm{L}}^{2}(\eb)}
\newcommand{\Le}{{\rm{L}}^{2}({\bf{e}})}
\newcommand{\lrp}{{\rm{L}}^{2}(\R_+)}
\newcommand{\co}{{\rm{C}}_0(\R_+)}
\newcommand{\ci}{{\rm{C}}^{\infty}(\R_+)}
\newcommand{\cb}{{\rm{C}}_b(\R_+)}
\newcommand{\bca}{\overline{C}_{\alpha}}
\newcommand{\bba}{\overline{B}_{\alpha}}
\newcommand{\baa}{\overline{A}_{\alpha}}
\newcommand{\ra}{\bar{\beta}_{\alpha}}
\newcommand{\rra}{\bar{\rho}_{\alpha}}
\newcommand{\ebb}{\overline{{\bf{e}}}_{\gamma,\beta,\alpha}}
\newcommand{\lnubb}{{\rm{L}}^{2}(\ebb)}
\newcommand{\wn}{\mathcal{W}_n}
\newcommand{\Rn}{\mathcal{R}_n}
\newcommand{\ov}{\overline{\varsigma}}
\newcommand{\bo}[1]{{\rm{O}}\lb{#1}\rb}
\newcommand{\ew}{\overline{\eta}_{\alpha}}

\newcommand{\Wr}{\Psi}

\renewcommand{\limsup}{\mathop{\overline{\lim}}\limits}

\newcommand{\M}{\mathcal{M}}

\newcommand{\IInf}{\int_{0}^{\infty}}

\newcommand{\ind}[1]{\mathbb{I}_{\{#1\}}}
\newcommand{\lb}{\left (}
\newcommand{\rb}{\right )}
\newcommand{\lbb}{\left [}
\newcommand{\rbb}{\right ]}
\newcommand{\labs}{\left |}
\newcommand{\rabs}{\right |}
\newcommand{\langlerangle}[1]{\left<#1\right>}
\newcommand{\lbrb}[1]{\lb #1 \rb}

\newcommand{\labsrabs}[1]{\labs#1\rabs}
\newcommand{\lbbrb}[1]{\lbb#1\rb}
\newcommand{\lbrbb}[1]{\lb#1\rbb}
\newcommand{\lbcurly}{\left\{}
\newcommand{\rbcurly}{\right\}}
\newcommand{\lbcurlyrbcurly}[1]{\lbcurly#1\rbcurly}
\newcommand{\dg}[1]{#1_*}
\newcommand{\Lg}{{\rm{L}^{2}({\e})}}

\newcommand{\Pns}{(\mathcal{P}_n)}
\newcommand{\Pon}{\mathcal{P}_n}

\newcommand{\ratio}[1]{\frac{#1}{#1+1}}

\newcommand{\simi}{\stackrel{\infty}{\sim}}

\newcommand{\var}{\vartheta_\mu}

\newcommand{\Cb}{\mathbb{C}}

    \def\beq{\begin{eqnarray}}
    \def\eeq{\end{eqnarray}}
    \def\beqq{\begin{eqnarray*}}
    \def\eeqq{\end{eqnarray*}}
    \def\C{{\mathbb C}}

    \def\N{{\mathbb N}}

    \def\R{{\mathbb R}}

\newcommand{\Tab}{T_{\alpha}}

\author{P. Patie}
\address{School of Operations Research and Information Engineering, Cornell University, Ithaca, NY 14853.}
\email{	pp396@cornell.edu}

\author{M. Savov} \thanks{This work was partially supported by the Actions de Recherches Concert\'ees IAPAS, a fund of the french community  of Belgium. The second author also acknowledges the support of the project MOCT, which has received funding from the European Union’s
Horizon 2020 research and innovation programme under the Marie Sklodowska-Curie grant
agreement No 657025.}
\address{Institute of Mathematics and informatics,  Bulgarian academy of sciences, Akad.  Georgi Bonchev street
	Block 8, Sofia 113.}
\email{mladensavov@math.bas.bg}
\title{Cauchy Problem of the non-self-adjoint Gauss-Laguerre semigroups and uniform bounds for generalized Laguerre polynomials}
\usepackage{hyperref}
\begin{document}
\begin{abstract}
We propose a new approach to construct the eigenvalue expansion in a weighted Hilbert space of the solution to the Cauchy problem associated to Gauss-Laguerre invariant Markov semigroups that we introduce. Their  generators turn out to be  natural non-self-adjoint and non-local generalizations of the Laguerre differential operator. Our methods rely on  intertwining relations that we establish between these semigroups and  the classical Laguerre semigroup and combine with techniques  based on non-harmonic analysis. As a by-product we also provide regularity properties for the semigroups as well as for their heat kernels.
The biorthogonal sequences that appear in their eigenvalue expansion can be  expressed in terms of  sequences of polynomials, and they  generalize the Laguerre polynomials. By means of a delicate saddle point method,  we derive uniform asymptotic bounds that allow us to get an upper bound for their norms in weighted Hilbert spaces.  	We believe that this work opens a way to construct spectral expansions for more general non-self-adjoint Markov semigroups.
\end{abstract}
\maketitle
%\tableofcontents

\section{Introduction and main results}
 For any  $\alpha \in (0,1)$ and $\beta\in [1-\frac{1}{\alpha},\infty)$, we define the Gauss-Laguerre operator as the linear integro-differential operator which takes the form, for a smooth function $f$ on $x>0$,
\begin{eqnarray}  \label{eq:EK_der}
\mathbf{L}_{\alpha,\beta}\:f(x) &=& \lb {\rm{d}}_{\alpha,\beta} -x\rb f'(x)+\frac{\sin(\alpha \pi)}{\pi}x\int_0^1 f''(xy) g_{\alpha,\beta}(y) dy,
\end{eqnarray}
where  $\rm{d}_{\alpha,\beta}= \frac{\Gamma( \alpha \beta+\alpha +1)}{\Gamma(\alpha \beta +1)}$ and
\begin{equation}
g_{\alpha,\beta}(y)=\frac{\Gamma(\alpha)}{\beta+\frac1\alpha+1}y^{\beta+\frac1\alpha+1} \: {}_{2}F_{1}(\alpha (\beta+1)+1,\alpha+1;\alpha (\beta+1)+2;y^{\frac{1}{\alpha}}),
\end{equation} with ${}_{2}F_{1}$ the Gauss hypergeometric function. The terminology is motivated by the limit case $\alpha=1$ which will be proved to yield
\begin{eqnarray*}  \label{eq:EK_deriv}
	\mathbf{L}_{\beta}f(x) &=&\mathbf{L}_{1,\beta} f(x)= x f''(x) + \lb \beta+1 -x\rb f'(x),
\end{eqnarray*}
that is the Laguerre differential operator of order $\beta$. It is well known to be the generator of a self-adjoint contraction semigroup $ (Q^{(\beta)}_t)_{t\geq0}$ in the weighted Hilbert space $\lnub$, where ${\bf{e}}_{\beta}={\bf{e}}_{1,\beta}$ is the density of the unique invariant measure  and the later is defined in \eqref{eq:def_e} below. This semigroup as well as its eigenfunctions, the  Laguerre polynomials, have been and are still intensively studied as they play a central role in probability theory, functional analysis, representation theory, quantum mechanics and mathematical physics, see e.g.~ \cite{Bakry_Book}, \cite{Kost}, \cite{Zet} and the references therein.
The Gauss-Laguerre semigroup, whose infinitesimal generator shares some similarities with the classical Caputo fractional derivative of order $\alpha$,  also appear in some recent applications in biology, see e.g.~\cite{GL_Fragm} and \cite{GL_Biology} and the references therein.
 Similarly to the classical Laguerre semigroup, we shall now prove  the following fact where $\mathcal{A}$ stands for the algebra  of polynomials.
\begin{theorem} \label{thm:main1}
 For any  $\alpha \in (0,1)$ and $\beta\in [1-\frac{1}{\alpha},\infty)$, $\mathbf{L}_{\alpha,\beta}$ is the generator with core $\mathcal{A}$ of a non-self-adjoint contraction Markov semigroup $P=(P_t)_{t\geq 0}$ in the Hilbert space ${\rm{L}}^2(\eab)$ endowed with the norm  $||f||_{\eab}=\int_0^{\infty}f^2(x)\eab(x)dx$ where
 \begin{equation}\label{eq:def_e}
 \eab(x)dx=\frac{x^{\beta+\frac1\alpha-1} e^{-x^{\frac1\alpha}}}{\Gamma(\alpha \beta +1)}dx, \: x>0,
 \end{equation}
  is the unique invariant  measure of $P$.
\end{theorem}	
%Note that we shall actually prove that  the  operator $\mathbf{L}_{\alpha,\beta}$ in the statement is  an extension of the  linear operator defined in \eqref{eq:EK_der} which acts (at least) on $\mathcal{A}$, which is a dense subspace of ${\rm{L}}^2(\eab)$.

The aims of this paper are  to provide (a)  a spectral representation in  the weighted Hilbert space $\lnua$ of the semigroup $(P_t)_{t\geq0}$, (b)  regularity properties of $P_tf$ for $f$ in various spaces, (c) an explicit representation and smoothness properties of the heat kernel (or the (density of) transition probabilities of the underlying Feller process). Note that this study allows to obtain an explicit representation and smoothness properties of the solution to the following Cauchy problem
\begin{equation} \label{eq:cauchy}
\begin{cases}
\frac{d }{d t} u_t(x) &= \mathbf{L}_{\alpha,\beta}\:u_t(x) \\
u_0(x) &= f(x) \in \mathcal{D},
\end{cases}
\end{equation}
where  $\mathcal{D}$ stands for the domain of $\mathbf{L}_{\alpha,\beta}$.
 %The main challenges underlying this study  are the lack of spectral theory for non-self-adjoint operators and the fact that our framework is one of a weighted Hilbert space.
There are several motivations underlying this work. On the one hand, although the spectral theory for linear self-adjoint, or more generally normal, operators is well established, see e.g.~\cite{Dunford_II}, the spectral properties of non-self-adjoint operators is fragmentally understood. We refer for instance to the survey papers of Davies \cite{Davies_S}  and Sj\"o{strand} \cite{Sjostrand-Survey} for a nice account of recent developments in this area.  There are very few instances in the literature where the spectral expansion of non-self adjoint linear operators is available. Among the notable exceptions are the integral operators characterizing the formal inverses of Wilson divided difference operators, studied by Ismail and Zhang \cite{Ismail-Zhang}, and, the harmonic oscillator,  arising in quantum mechanics, and acting on ${\rm{L}}^2(\R_+)$, whose study has been initiated by Davies in \cite{Davies} and further developed by Davies and Kuijlaars \cite{Davies-Kuijlaars}. In the framework of Markov semigroups,  the spectral expansion of one dimensional self-adjoint diffusion  was developed by McKean \cite{McKean-Spectral}, and extended by Getoor \cite{Getoor_Spectral}, to some non local self-adjoint operators. Although  non-self-adjoint operators seem to be generic in the class of Markov semigroups, we are not aware of any results concerning the spectral representation in Hilbert space of a non-self-adjoint positive contraction semigroup.  On the other hand, the Gauss-Laguerre semigroup  turns out to play an essential role in the recent work by the authors \cite{Patie-Savov-GeL} concerning  the spectral expansion of a large class of non-self-adjoint invariant Markov semigroups. This class can be either characterized in terms of the generator which   takes the form of a linear combination (with non negative coefficients) of $\mathbf{L}_{\beta}$ and $\mathbf{L}_{\alpha,\beta}$, where for the later the function $g_{\alpha,\beta}$ can be any positive convex functions satisfying a mild integrability condition. Another characterization could be made through a bijection that we established between this class of semigroups and the set of Bernstein functions, which appears in the action of the generator on monomials, as in \eqref{eq:action_pol} below, with the Bernstein function $\Phi_{\alpha,\beta}$. In the aforementioned paper, the Gauss-Laguerre semigroup serves as a reference semigroup, via an intertwining relation, with  the class of semigroups associated to regularly varying Bernstein functions. This  concept of reference semigroups allows for instance to obtain estimates for the norms of the co-eigenfunctions of seemingly intractable operators.

Coming back to the present work, it aims at presenting a new methodology, which contains some comprehensive idea, for developing the spectral expansion of the Markov semigroup $(P_t)_{t\geq0}$ thus opening the possibility to understand better the spectral expansions of more general Markov semigroups. Our first main idea is to derive an intertwining relation, via a Markov operator,  between the class of non-self-adjoint Gauss-Laguerre semigroups and the classical Laguerre semigroup of order $0$.  We say  that a linear operator $\Lambda_{\lambda}\:$ is a Markov operator if, for any $f \in B_b(\R_+)$, the set of bounded Borel functions on $\R_+$,
\begin{equation} \label{eq:def_MK} \Lambda_{\lambda}\: f(x)=\int_0^{\infty}f(xy)\lambda(y)dy, \end{equation}
where $\lambda$ is the density of a probability measure, i.e.~$\lambda\geq 0$ and $\int_{0}^{\infty}\lambda(y)dy=1$.  More specifically,  defining  the entire function $\lambda_{\alpha,\beta}$ by
	\begin{eqnarray} \label{eq:def_lambda}
		\lambda_{\alpha,\beta}(z) &=& \frac{\Gamma(\alpha \beta+1-\alpha)}{\pi}\sum_{k=0}^{\infty} \Gamma(\alpha k + \alpha(1- \beta)) \sin\lb \alpha(k+1- \beta)\pi\rb \frac{z^k}{k!}, \: z\in\C,
	\end{eqnarray}
 we have the following result, with the notation   $\Lambda_{{\alpha,\beta}}=\Lambda_{\lambda_{\alpha,\beta}}$, $\e = \e_{1,0}$ and where $(Q_t)_{t\geq0} = (Q^{(0)}_t)_{t\geq0} $ stands for the Laguerre semigroup of order $0$.
\begin{theorem} \label{thm:main2} 	
$\Lambda_{{\alpha,\beta}}:\Lg \mapsto \lnua$ is a one-to-one bounded Markov  operator with a dense range, i.e.~$\overline{{\rm{Ran}}}(\Lambda_{{\alpha,\beta}})=\lnua$. Moreover, for any $t\geq0$, the intertwining relation
		\begin{equation}\label{eq:inter}
		P_t \: \Lambda_{\alpha,\beta}  = \Lambda_{\alpha,\beta}\: Q_t
		\end{equation}
	holds on  $\Lg$.
\end{theorem}
\begin{remark}
	\begin{enumerate}
			\item Although, by means of the Marcinkiewicz multiplier theorem for Mellin transform, see \cite{Rooney-Muck},  it is an  easy exercise to show, from the asymptotic behavior of its Mellin multiplier, see \eqref{eq:def_Mellin_X} below, that a Markov operator is bounded from $\rm{L}^2(\var)$, $\var(x)=x^{-\alpha},x>0$, into itself, the continuity property on a weighted Hilbert spaces is in general a difficult problem. One classical approach is to consider weights which belong to the so-called class of Muchkenboupt, conditions which are not satisfied by $\e$. Instead, we identify a factorization of Markov operators which allows to derive by a simple application of Jensen inequality the contraction property.
		\item With the aim of developing the spectral expansion of the semigroup $P$, we mention that the intertwining relation \eqref{eq:inter}  goes beyond perturbation theory. Indeed, clearly $\mathbf{L}_{\alpha,\beta}$  is by no means a perturbation of a self-adjoint operator whereas the relation \eqref{eq:inter} relates it to a self-adjoint operator.
	\end{enumerate}
\end{remark}		
We shall exploit the intertwining relation to develop the spectral representation of $(P_t)_{t\geq0}$. Although the literature on intertwining relations between Markov semigroups and its applications is very rich, see for instance Dynkin \cite{Dynkin-69}, Pitman and Rogers \cite{Pitman-Rogers-81} and Carmona et al.~\cite{Carmona-Petit-Yor-98}, it does not seem that it has served for this purpose.  On the other hand, this type of commutation relation between linear operators have been
also intensively studied  in the context of differential operators. This
approach culminated in the work of Delsarte and Lions \cite{Delsarte-Lions-57} who showed the existence of a
transmutation operator between differential operators of the same order and acting on the
space of entire functions. The transmutation operator, which plays the role of the intertwining
operator, is in fact an isomorphism on this space. This property is very useful for the spectral
reduction of these operators since it allows to transfer the spectral objects. We mention
that Delsarte and Lions's development has been intensively used in scattering theory and in
the theory of special functions, see e.g.~Carroll and Gilbert [17]. We shall prove that our intertwining operator is not bounded from below, a property which makes the analysis of the spectral expansion more delicate than in the framework of transmutation operators. To overcome this difficulty, we resort to the concept of frames, a generalization of orthogonal sequences that has been introduced by
Duffin  and Schaeffer \cite{Duffin-Schaeffer-52}  to study some deep problems in non-harmonic Fourier series.
Next, we recall that,  by means of the spectral theory for self-adjoint operators, one obtains, for any $f \in \lnub$ and $t>0$, the classical spectral expansion
 \begin{equation}\label{eq:expansionLaguerre}
 Q^{(\beta)}_t f(x) = \sum_{n=0}^{\infty}e^{-n t} \langle f,\mathcal{L}^{(\beta)}_n \rangle_{\eb} \:  \overline{\beta}^{-2}_n\: \mathcal{L}^{(\beta)}_n(x) \qquad \textrm{ in } \lnub,
 \end{equation}
 where $\overline{\beta}^{2}_n = \frac{\Gamma(n+1)\Gamma(\beta+1)}{\Gamma(n+\beta+1)}$,  $\mathcal{L}^{(\beta)}_n$ is the Laguerre polynomial of order $\beta$ defined as
 \begin{equation} \label{eq:def_lag_b}
 \overline{\beta}^2_n \: \mathcal{L}^{(\beta)}_n(x)= \sum_{k=0}^n (-1)^k \frac{{ n \choose k}}{{ k+\beta \choose \beta}} \frac{x^k }{k!}= \Gamma(\beta+1)\sum_{k=0}^n (-1)^k  \frac{{ n \choose k}}{\Gamma(k+\beta+1)}x^k,
 \end{equation}
 and, the sequence $(\overline{\beta}_n \mathcal{L}^{(\beta)}_n)_{n\geq 0}$ is an orthonormal  sequence in $\lnub$. Before stating the next result, we proceed with some further notation. For any $x\geq0$, we set $\mathcal{P}_0(x)=1$ and for any $n\geq 1$, we introduce the polynomials
\begin{equation}\label{eq:Pn}
 \mathcal{P}_n(x) =\Gamma(\alpha \beta +1) \sum_{k=0}^n (-1)^k\frac{{ n \choose k}}{\Gamma(\alpha k + \alpha \beta +1)} x^k.
 \end{equation}
Note that for $\alpha=1$, $\mathcal{P}_n(x) = \overline{\beta}^2_n\mathcal{L}_n^{(\beta)}(x)= \Gamma(\beta +1) \sum_{k=0}^n (-1)^k\frac{{ n \choose k}}{\Gamma( k + \beta +1)} x^k$ is the classical Laguerre polynomial of order $\beta\geq 0$. Moreover,
 for any  $ x\geq0$ and $n\in\N$,  we write  \begin{equation}
 \mathcal{R}_n(x)={\rm{R}}^{(n)}_{\eab} \eab(x)=\frac{(-1)^n}{n!\eab(x)} (x^{n}\eab(x))^{(n)},
 \end{equation}
 where   $\rm{R}^{(n)}_{\eab}$ is the weighted Rodrigues operator and  $f^{(n)}=\frac{d^n}{dx^n}f$. From  the Rodrigues representation of the Laguerre polynomials, we also get that for $\alpha=1$, $\mathcal{R}_n(x)=\mathcal{L}_n^{(\beta)}(x)$. Finally, we  define, for any $0<\gamma<\alpha$ and $\ew>0$ fixed,
 \begin{equation}
 \ebb(x)=x^{\beta+\frac1\alpha-1} e^{\ew x^{\frac1\gamma}}, \: x>0,
 \end{equation}
 where we recall that $\alpha \in (0,1)$ and $\beta\in [1-\frac{1}{\alpha},\infty)$, and  set \[\Tab =-\ln\lbrb{2^{\alpha}-1}.\] We are now ready to state the main result of the paper. 
\begin{theorem} \label{thm:main}
	\begin{enumerate}[(a)]
\item For any $f \in \lnua$ (resp.~$f  \in {\rm{Ran}}(\Lambda_{\alpha,\beta}) \cup \lnubb $)  %(resp.~$f \in \lnubb$ improve here below),
we have
\begin{equation} \label{eq:exp}
P_t f(x) = \sum_{n=0}^{\infty} e^{-nt} \spnu{f,\mathcal{R}_n}{\eab} \mathcal{P}_n(x),
\end{equation}
where, for any $t>\Tab$ (resp.~$t>0$), the  identity holds  in  $\lnua$. $P_t$ is holomorphic in $t$ on $\C_{(\Tab,\infty)}=\{z\in \C;\: \Re(z)>\Tab\}$.
\item For any  $f \in \lnua$ (resp.~$f  \in {{\rm{Ran}}}(\Lambda_{\alpha,\beta}) \cup \lnubb $), $ (t,x)\mapsto P_t f(x) \in {\rm{C}}^{\infty}\lb (T_\alpha,\infty) \times \R_+ \rb$ (resp.~$\in {\rm{C}}^{\infty}(\R^2_+)$), and for any integers $k,p$,
\begin{eqnarray*} \label{eq:exp_derv}
	\frac{d^k}{dt^k}(P_t f)^{(p)}(x) =\sum_{n=p}^{\infty}(-n)^k e^{-n t} \langle  f,\mathcal{R}_n \rangle_{\eab} \:  \: \mathcal{P}^{(p)}_{n}(x)
	\end{eqnarray*}
	where, for any $t>\Tab$ (resp.~$t>0$), the series converges locally uniformly on $\R_+$.

\item The heat kernel is absolutely continuous with  a density $ (t,x,y) \mapsto P_t(x,y) \in {\rm{C}}^{\infty}(\R^3_+)$, given for  any $t,y>0$,  $x\geq0$, and for any integers $k,p,q$, by
 \begin{eqnarray} \label{eq:exp_dervallr}
 \frac{d^k}{dt^k}P^{(p,q)}_t(x,y) =\sum_{n=p}^{\infty}(-n)^k e^{-n t} \mathcal{W}^{(q)}_n(y)   \: \mathcal{P}^{(p)}_{n}(x),
 \end{eqnarray}
 where the series is locally uniformly  convergent on $\R^3_+$, and, for $n\geq0$, $\mathcal{W}_n(y)=\mathcal{R}_n(y)\eab(y)$.
\item $(P_t)_{t\geq0}$ is   a strong Feller semigroup, i.e.~ for any $t>0$ and $f \in B_b(\R_+)$, $P_tf \in \cb $, where $\cb$ is the space of bounded continuous functions on $\R_+$.
\end{enumerate}
\end{theorem}
\begin{remark}
	\begin{enumerate}[1)]
		\item The phenomenon that the expansion in the full Hilbert space holds only for  $t$ bigger than a constant has been observed in the framework of Schr\"odinger operator, see \cite{Davies} and is natural for non-normal operators. Indeed, in such a case, the spectral projections $P_n f = \langle  f,\mathcal{R}_n \rangle_{\eab} \:   \mathcal{P}_{n} $  are not uniformly bounded as a sequence of operators.  The projections are not orthogonal anymore and the sequence of eigenfunctions does not form a basis of the Hilbert space.   These two facts illustrate  a fundamental difference with self-adjoint Markov semigroups, for which the spectral projections are orthogonal and uniformly bounded.
		\item In order to provide the convergence of the expansion \eqref{eq:exp} in the Hilbert space topology, we rely on the so-called synthesis operator as defined  in \eqref{eq:def_synt} below which  requires to characterize those $f$ and $t$ for which  the sequence $( e^{-nt}\langle  f,\mathcal{R}_n \rangle_{\eab}) \in \ell^2(\N)$. This is a difficult problem in general.  A  natural approach to verify this property is to resort to the Cauchy-Schwarz inequality which yields, thanks to the first bound stated in Proposition \ref{thm:bound}, to the description of $\Tab$, the smallest $t$ for which the expansion holds. From this perspective, we also manage to identify the Hilbert space $\lnubb$ for which the expansion is valid for all $t>0$, an approach which seems to be original in this context.
		\item
		Moreover, it is worth pointing out that the intertwining approach  enables to identify ${{\rm{Ran}}}(\Lambda_{\alpha,\beta})$ as another linear space for which the corresponding sequence is in $\ell^2(\N)$ for all $t>0$, a property which follows  directly without using  any bounds. In fact, we shall have the stronger statement  that for any $ f \in {{\rm{Ran}}}(\Lambda_{\alpha,\beta})$, $f= \sum_{n=0}^{\infty} \spnu{f,\mathcal{R}_n}{\eab} \mathcal{P}_n$. Finally, as we shall prove that $\mathcal{A} \subset{{\rm{Ran}}}(\Lambda_{\alpha,\beta})$ whereas for any $n\geq0$, $\mathcal{P}_n \notin \lnubb$,  we are lead to think that either our optimal Hilbert space may be improved or the Cauchy-Schwarz inequality provides weak estimate in our scenario. However, from the biorthogonality property \eqref{eq:def_bio},  we believe that the latter explanation is in force in this context.
		\item Finally, we shall prove in Lemma \ref{lem:sem} that there exists $(K_t)_{t\geq0}$ a  $1$-selfsimilar Feller semigroup on $\R_+$, i.e.~for any $c>0$, $K_{ct}f(cx)=K_t f\circ d_c(x)$ with $d_cf(x)=f(cx)$, such  that,  for any $t\geq0$, $P_t f(x) = K_{e^t -1} f \circ d_{e^{-t}}(x)$. Note that  $(K_t)_{t\geq0}$ belongs to the class of semigroups introduced by Lamperti \cite{Lamperti-72} which play a central theorem in limit theorems of stochastic processes, see \cite{Lamperti-62}. In particular, one obtains from \eqref{eq:exp_dervallr} that $(K_t)_{t\geq0}$ has an absolutely continuous kernel, $K_t(x,y)$ given, for any $t,y>0$, $x\geq0$, by
		\begin{eqnarray*}
			K_t(x,y) =\sum_{n=0}^{\infty}(1+t)^{-n-1} \wn\lb\frac{y}{1+t}\rb\mathcal{P}_{n}(x).
		\end{eqnarray*}
	\end{enumerate}
\end{remark}

The remaining part of the paper is organized as follows. In the next Section, we state several substantial results regarding properties of the sequence of (co)-eigenfunctions  which some of them may have independent interests. Section \ref{sec:prel} gathered some useful preliminaries results and sections \ref{sec:proof2} to  \ref{sec:proof1} contain the proof of the main results.  Note that Section \ref{sec:bound} which includes the proof of Proposition \ref{thm:bound} below presents several uniform asymptotic estimates of $|\wn(x)|$ which might also be of independent interests.

\section{Substantial auxiliary results}
We start by  stating several interesting properties that the sequences $(\mathcal{P}_n)$ and $(\mathcal{R}_n)$  satisfy. For this purpose, we introduce some concepts  borrowed from non-harmonic analysis which are nicely exposed  in the monographs \cite{Young} and \cite{Christensen-03}.  Two sequences  $(\mathcal{P}_n)$ and $(\mathcal{R}_n)$ are said to be biorthogonal  in $\lnua$ if for any $n,m \in \N$,
\begin{equation} \label{eq:def_bio}
\spnu{\mathcal{P}_n,\mathcal{R}_m}{\eab} = \delta_{nm}.
\end{equation}
Moreover, a sequence that admits a biorthogonal sequence will
be called \emph{minimal} and a sequence that is both minimal and complete, in the sense that its linear span is dense in $\lnua$, will be called \emph{exact}.  It is easy to show that a sequence $(\mathcal{P}_n)$  is minimal if and only if none of its elements can be approximated by linear combinations of the others. If this is the
case, then a biorthogonal sequence will be uniquely determined if and only if $(\mathcal{P}_n)$
is complete. We also say that $(\mathcal{P}_n)$ is a Riesz basis in $\lnua$ if there exists an isomorphism $\Lambda$ from $\Le$ onto $\lnua $ such that $\Lambda \mathcal{L}_n= \mathcal{P}_n$ for all $n$.

\begin{proposition}\label{prop:sequence}
	\begin{enumerate}[1)]
		\item For any $n\in \N$, $\mathcal{P}_n \in \lnua$ and $\mathcal{R}_n \in \lnua$.
%	\item We have $\overline{span}(\mathcal{P}_n)=\overline{span}(\mathcal{R}_n)=\lnua$.
	\item The sequences $(\mathcal{P}_n)$ and $(\mathcal{R}_n)$ are biorthogonal and exact in $\lnua$.
\item \label{item:Bessel} Finally the sequences   $(\mathcal{P}_n)$  is not a Riesz basis but it  satisfies the following Bessel inequality
	\begin{equation}
	\sum_{n=0}^{\infty}|\spnu{f,\mathcal{P}_n}{\eab}|^2 \leq ||f||_{\eab}, \quad \forall f \in \lnua.
	\end{equation}
\end{enumerate}
\end{proposition}
 An interesting consequence of \ref{item:Bessel}) is the fact that the  synthesis operator $S$ defined by
 \begin{equation} \label{eq:def_synt}
 S(l_n) = \sum_{n\geq0} l_n \mathcal{P}_n
 \end{equation}
is bounded from  $ \ell^2(\N)$ into  $\lnua$ with  $||S(l_n)||^2_{\eab}\leq\sum_{n\geq0} l^2_n $, and, for such a sequence, the series converges unconditionally. Although this information is very helpful for our purpose, one still needs estimates for large $n$ of  $||\mathcal{R}_n||_{\eab}$, $|\mathcal{R}_n(x)|$ and  $|\mathcal{P}_n(x)|$  in order to derive the convergence properties of the eigenvalue expansions in the appropriate topology.
We state the following bounds for the two latter quantities.
\begin{proposition} \label{prop:crude_bound}
\begin{enumerate}
	\item
 Writing $\mathfrak{t}_{\alpha}=(\alpha+1)\alpha^{-\ratio{\alpha}}$,
we have for any $x\in \R$, any integer $p$, and, $n$ large
\begin{equation} \label{eq:asympt_polyn}
|\mathcal{P}^{(p)}_{n}(x)| ={\rm{O}}\lb n^{p+\frac12} e^{\mathfrak{t}_{\alpha}(n|x|)^{\frac{1}{1+\alpha}}} \rb.
\end{equation}
\item  Writing  $\bar{\mathfrak{t}}_{\alpha}=\mathfrak{t}_\alpha\lb \frac{\alpha+1}{\alpha}+\epsilon\rb^{\frac{1}{\alpha+1}}$, for some small $\epsilon>0$, we have, for any $0<x< e^{-2\alpha} \lbrb{\ratio{\alpha}}^{\alpha}n^\alpha$, any integer $q$,  and  large $n$
\begin{eqnarray}\label{eq:crude_bound}
\left|\mathcal{W}^{(q)}_n\lb x \rb \right| =\bo{x^{\beta + \frac{1}{\alpha}-q}n^{\labs\beta + \frac{1}{\alpha}-1-q\rabs+2} e^{\bar{\mathfrak{t}}_{\alpha} \lbrb{nx}^{\frac{1}{\alpha+1}}}}.
\end{eqnarray}
\end{enumerate}	
\end{proposition}
Next, we recall that when $\alpha=1$, i.e.~$\mathcal{R}_n$ is simply the classical Laguerre polynomials, one uses the following simple observation to compute their norms, see e.g.~\cite{Szego},
\begin{eqnarray*}
	||\mathcal{L}_n^{(\beta)}||^2_{\eb} &=& \frac{1}{\Gamma(\beta +1)}  \int_0^{\infty}(\mathcal{L}_n^{(\beta)}(x))^2x^{\beta} e^{-x}dx = \frac{(-1)^n}{n!} \int_0^{\infty}\mathcal{L}_n^{(\beta)}(x) (x^{n}\eb(x))^{(n)}dx \\
	&=&\frac{1}{n!} \int_0^{\infty}x^{n}\eb(x)dx =\frac{\Gamma(n+\beta +1)}{n!\Gamma(\beta  +1)}.
\end{eqnarray*}
Unfortunately, it is easy to check that for $\alpha \in (0,1)$, this integration by parts device does not apply.  Instead, we must develop a two-steps optimization analysis to derive the estimates of the norms. First, we carry out  delicate saddle point approximations  to obtain several uniform bounds for $|\mathcal{R}_n(x)|$ depending on the range of $xn^{-\alpha}$, and, refer to Proposition \ref{thm:Bounds} for their statements. In this vein, we mention that the study of uniform asymptotic expansions of  the Laguerre polynomials has  quite a long history, see e.g.~\cite{Frenzen-Wong}, \cite{Olver_W} and also \cite{Szego} and \cite{Temme} for a complete description of this study. Then, combining these bounds with additional estimates, we must implement   a suboptimal procedure in order to get an explicit representation of the bound of their $\lnua$-norm. Moreover, although for most of the ranges one may obtain bounds of the form  $\bo{e^{\epsilon n}}$ for any $\epsilon>0$, for larger Hilbert spaces than $\lnubb$, it turns out that on the range $x \in (\epsilon n^{\alpha}, \bca n^{\alpha})$ for some constant $\bca$ defined in Proposition \ref{thm:Bounds},  $\lnubb$ is the optimal Hilbert space.
 From our analysis, we obtain the following estimates.
\begin{proposition}\label{thm:bound}
	We have for large $n$,
	\begin{equation}
	\label{eq:bound_norm}
	||\mathcal{R}_n||_{\eab} = \bo{e^{ \Tab n}},
	\end{equation}
	and  %$\bar{\mathfrak{t}}_{\alpha}=(\alpha+1)\alpha^{-\ratio{\ %alpha}}\lb %\frac{\alpha+1}{\alpha}+\epsilon\rb^{\frac{1}{\alpha+1}} %$, we have for large $n$,
	\begin{equation}
	\label{eq:bound_norm_new}
	\left|\left|\Rn\frac{\eab}{\ebb}\right|\right|_{\ebb} = \bo{n^{1+\beta+\frac1\alpha+\alpha}e^{\bar{\mathfrak{t}}_{\alpha}n^{\frac{1}{\alpha+1}}}}.
	\end{equation}
\end{proposition}
\section{Some preliminary results} \label{sec:prel}
\subsection{Some useful facts around the gamma function}
Let us write,
for any  $\alpha \in (0,1]$ and $\beta\geq 1-\frac{1}{\alpha}$, and $\Re(s)>-\beta-\frac{1}{\alpha}$,
\begin{equation}
\label{eq:def_phir}
\Phi_{\alpha,\beta}(s)=
\frac{\Gamma(\alpha s + \alpha \beta +1)}{\Gamma(\alpha s +\alpha \beta +1-\alpha)}.
\end{equation}
In the following we collect some basic results which will be useful throughout the rest of the paper.
\begin{lemma}
\begin{enumerate}
	\item For any  $\alpha \in (0,1]$ and $\beta\geq 1-\frac{1}{\alpha}$, and $k\geq 1$, we have
		\begin{equation}
		\label{eq:poly}
		\frac{\sin(\alpha \pi)}{\pi}k(k-1)\int_0^1 y^{k-2} g_{\alpha,\beta}(y) dy = k\Phi_{\alpha,\beta}(k)-k\:{\rm{d}}_{\alpha,\beta}.
		\end{equation}
	\item For any $\Re(s)>-\beta-\frac{1}{\alpha}$, the functional equation
	\begin{equation*}
	%\label{eq:fe_phir}
	\frac{\Gamma(\alpha s+\alpha\beta+1)}{\Gamma(\alpha  \beta+1)}=\Phi_{\alpha,\beta}(s)\frac{\Gamma(\alpha s+\alpha \beta+1-\alpha)}{\Gamma(\alpha  \beta+1)}
	\end{equation*}
	holds.
	\item Finally, we have, for large $|b|$ and $|\arg(a+ib)|<\pi$, the following well-known classical asymptotic estimates %, with $z=a+ib$,
	\begin{eqnarray}\label{eqn:RefinedGamma1}
		|\Gamma(a+ib)|&=& Ce^{-a}e^{a\ln{|a+ib|}}e^{-b \arg(a+ib)}|a+ib|^{-\frac{1}{2}}(1+o(1)) ,\\
	|\Gamma(a+ib)|&\sim& C_a |b|^{a-\frac{1}{2}}e^{-\frac{\pi}{2}|b|}, \label{eq:est_gamma}
	\end{eqnarray}
	where $C,C_a>0$.
\end{enumerate}	

\end{lemma}
\begin{proof}
First, observe, from the binomial formula, that,  for any $0<y<1$,
	\begin{eqnarray}
	\int_0^y \frac{r^{\beta+\frac1\alpha} }{(1-r^{\frac{1}{\alpha}})^{\alpha+1}} dr  &=&  \sum_{k=0}^{\infty}\frac{\Gamma(k+\alpha+1)}{\Gamma(\alpha+1)k!} \int_0^y r^{\beta+\frac1\alpha +\frac{k}{\alpha}} dr  \nonumber  \\
	& =&\alpha y^{\beta+\frac1\alpha+1} \sum_{k=0}^{\infty}\frac{\Gamma(k+\alpha+1)}{(k+\alpha(\beta+1)+1) \Gamma(\alpha+1)} \frac{y^{\frac{k}{\alpha}}}{k!}  \nonumber \\
	& =& \alpha y^{\beta+\frac1\alpha+1} \sum_{k=0}^{\infty}\frac{\Gamma(k+\alpha(\beta+1)+1) \Gamma(k+\alpha+1)}{\Gamma(k+\alpha(\beta+1)+2) \Gamma(\alpha+1)} \frac{y^{\frac{k}{\alpha}}}{k!} \nonumber \\
	&=&
	\frac{y^{\beta+\frac1\alpha+1}}{\beta+\frac1\alpha+1} \: {}_{2}F_{1}\lb \alpha (\beta+1)+1,\alpha+1;\alpha (\beta+1)+2;y^{\frac{1}{\alpha}}\rb.  \label{eq:EK_deriv1}
	\end{eqnarray}
	Then, by integration by parts and using the reflection formula of the gamma function, we get
		\begin{eqnarray}
		\label{eq:poly_s1}
		\frac{\sin(\alpha \pi)}{\pi}(k-1)\int_0^1 y^{k-2} g_{\alpha,\beta}(y) dy &=& \frac{1}{\Gamma(1-\alpha)}\int_0^1 (1-y^{k-1}) \frac{y^{\beta+\frac1\alpha} }{(1-y^{\frac{1}{\alpha}})^{\alpha+1}} dy. \end{eqnarray}
Next, from the integral representation of the Beta function, we get, for any $\alpha \in (0,1)$ and $u\geq0$,
	\begin{equation*}
	\frac{\Gamma(\alpha u +\alpha)}{\Gamma(\alpha u )} =\frac{1}{\Gamma(1-\alpha)}\int_0^{1}(1-y^u) (1-y^{1/\alpha})^{-\alpha-1}dy.
	\end{equation*}
	By shifting $u$ to $u+\beta+\frac1\alpha$, we get, after some easy algebra, that
	\begin{equation*}
	\frac{\Gamma(\alpha u + \alpha (\beta+1) +1)}{\Gamma(\alpha u +\alpha \beta +1)}-\frac{\Gamma( \alpha (\beta+1) +1)}{\Gamma(\alpha \beta +1)}= \frac{1}{\Gamma(1-\alpha)}\int_0^{1}(1-y^u) \frac{y^{ \beta+\frac{1}{\alpha}}}{(1-y^{1/\alpha})^{\alpha+1}} dy.
	\end{equation*}
	Thus choosing $u=k-1$, with $k\geq1$, in this latter identity, from \eqref{eq:poly_s1}, we deduce that
	\begin{eqnarray*}
	%\label{eq:poly_s2}
	\frac{k-1}{\Gamma(1-\alpha)}\int_0^1 y^{k-2} g_{\alpha,\beta}(y) dy &=& 	\frac{\Gamma(\alpha k + \alpha \beta +1)}{\Gamma(\alpha k +\alpha \beta +1-\alpha)}-\frac{\Gamma( \alpha (\beta+1) +1)}{\Gamma(\alpha \beta +1)}
	\end{eqnarray*}
which completes the proof of the first statement. The second one is obvious from \eqref{eq:def_phir}. The last estimates are readily deduced from the Stirling's formula, see e.g.~\cite[(2.1.8)]{Paris01},
\[
|\Gamma(z)|=C|e^{-z}||z^z||z|^{-\frac{1}{2}}(1+o(1)) \]
which is valid for large $|z|$ and $|\arg(z)|<\pi$.
\end{proof}

\subsection{The  Markov operator $\Lambda_{\alpha,\beta}$}
We recall, from \eqref{eq:def_MK}, that  a linear operator $\Lambda$ is a Markov operator if it admits the  representation,  for  any $ f \in B_b(\R_+)$,  $\Lambda  f(x)=\int_0^{\infty}f(xy)\lambda(y)dy, x>0,$ with  $\lambda$ the density of a probability measure.
We say  that $ \M_{\Lambda\:}=\M_{\lambda}$ is a Markov multiplier if  for $\Re(s)=0$,
\[ \M_{\lambda}(s)=\int_0^{\infty}y^{s}\lambda(y)dy, \]
that is,  the shifted Mellin transform of the density $\lambda$.
\begin{proposition} \label{prop:Kernel}
	Let  $\alpha \in (0,1)$ and $\beta\in [1-\frac{1}{\alpha},\infty)$ and define for any $\Re(s)=0$,
	\begin{equation} \label{eq:def_Mellin_kernel}
\log \M_{\lambda_{\alpha,\beta}}(s)= -\gamma_{\phi}(1- \alpha) s + \int_{-\infty}^0 \lb e^{sy}-1 -sy \rb \frac{(e^{-y}-1)^{-1}-e^{(\beta+\frac{1}{\alpha})y}(e^{-\frac{y}{\alpha}}-1)^{-1}}{|y|}dy.
	\end{equation}
	Then the following holds.
	\begin{enumerate}
		\item $\M_{\lambda_{\alpha,\beta}}$  is a Markov multiplier which is analytical on $\C_{(-1,\infty)}$.  $\lambda_{\alpha,\beta} \in {\rm{L}}^{2}(\R_+)$ and extends to an entire function which  admits the representation \eqref{eq:def_lambda}.
	\item   $e^y\lambda_{\alpha,\beta}(e^y)$ is the density of a real-valued infinitely divisible random variable.
	\item $\Lambda_{\alpha,\beta}\:$ is a contraction  from $\Lg$ into $\lnua$ with  $\overline{{\rm{Ran}}}(\Lambda_{\alpha,\beta})\: = \lnua$.
\end{enumerate}
\end{proposition}
%\begin{remark}
%	\begin{enumerate}
%		\item The proof of the one-to-one property of the Markov kernel will be provided below until which it will not played any significant role. We stated it here for sake of clarity.	
%	\end{enumerate}
%\end{remark}
\begin{proof}
Writing $h(y)=\left((e^{-y}-1)^{-1}-e^{(\beta+\frac{1}{\alpha})y}(e^{-\frac{y}{\alpha}}-1)^{-1}\right)\mathbb{I}_{\{y<0\}},$	one easily checks that $h(y)\geq0$ on $\R_-$ with $\int_{-\infty}^{0}(1\wedge y^2) \frac{h(y)}{y}dy<\infty$, that is $\frac{h(y)}{y}dy$ is a L\'evy measure and the right-hand side of \eqref{eq:def_Mellin_kernel} is the L\'evy-Khintchine exponent of an infinitely divisible random variable on the real line, see e.g.~\cite{Sato-99}. After performing  a change of variables and with the absolute continuity of its distribution which will be proved below, the second statement follows. Next, since from   \cite[1.9(1) p.21]{Erdelyi-55}, we have, for any $\Re(s)>0,$
	 \begin{equation*} %\label{eq:def_Mellin_Gamma}
	 \log\Gamma(s+1)= -\gamma_{\phi} s + \int_{-\infty}^0 \lb e^{sy}-1 -sy \rb \frac{(e^{-y}-1)^{-1}}{|y|}dy
	 \end{equation*}
	 we get, after some easy manipulations, that
	 \begin{equation*} %\label{eq:def_Mellin_Gamma_shift1}
	 \log\frac{\Gamma(\alpha (s+ \beta+\frac{1}{\alpha}))}{\Gamma(\alpha \beta+1)}= -\alpha \gamma_{\phi} s + \int_{-\infty}^0 \lb e^{sy}-1 -sy \rb \frac{e^{(\beta+\frac{1}{\alpha})y}(e^{-\frac{y}{\alpha}}-1)^{-1}}{|y|}dy.
	 \end{equation*}
	 The last two expressions easily lead to
	 \begin{equation} \label{eq:def_Mellin_X}
	 \M_{\lambda_{\alpha,\beta}}(s)=\frac{\Gamma(s+1)\Gamma(\alpha  \beta+1)}{\Gamma(\alpha (s+\beta+\frac{1}{\alpha}))}.
	 \end{equation}
	Hence since by  assumption $\beta+\frac1\alpha\geq1$, we have that $s\mapsto \M_{\lambda_{\alpha,\beta}}(s)$ is analytical on $\C_{(-1,\infty)}$. Moreover, for any $\epsilon>0$ and $|b|$ large and $a>-1$, we deduce from \eqref{eq:est_gamma}, that
	 \[ |\M_{\lambda_{\alpha,\beta}}(a+ib)|\leq C_a e^{-(1-\alpha -\epsilon)\frac{\pi}{2}|b|},\] with $C_a>0$. Thus,  on the one hand, since $|\M_{\lambda_{\alpha,\beta}}(-\frac{1}{2}+ib)| \in \lrp$, we deduce from the discussion above combined with the Parseval identity for Mellin transform that $\M_{\lambda_{\alpha,\beta}}(s-1)$ is the Mellin transform of a positive random variable whose law is absolutely continuous with a density in $\lrp$.  On the other hand,  by Mellin inversion, we get that, for  any $|\arg(z)|<(1-\alpha)\frac{\pi}{2}$,
	\begin{eqnarray*} \label{eq:def_Mellin_Gamma_shift}
	\lambda_{{\alpha,\beta}}(z)&=&\frac{1}{2\pi i}\int_{a-i\infty}^{a+i\infty}z^{-s}\frac{\Gamma(s)\Gamma(\alpha  \beta+1)}{\Gamma(\alpha s+ \alpha(\beta-1) +1)}ds\\
&=&\Gamma(\alpha (\beta-1)+1)\sum_{k=0}^{\infty} \frac{1}{\Gamma(-\alpha k + \alpha(\beta-1)+1)}\frac{z^k}{k!},
	\end{eqnarray*}
	where the last line follows from a classical application of the Cauchy residue theorem and we refer to \cite{Paris01} for more details on Mellin-Barnes integrals. An application of the reflection formula   provides the expression of $\lambda_{\alpha,\beta}$, i.e.~\eqref{eq:def_lambda}, whereas the Stirling approximation gives that the series is absolutely convergent on $\C$. Next, observe that	for any $\Re(z)>0$, we have
	\begin{eqnarray}\label{eq:MellinIdentity}
	\M_{\lambda_{\alpha,\beta}}(s)\M_{{\bf{e}}_{\alpha,\beta}}(s) & = & \frac{\Gamma(s +1)\Gamma(\alpha \beta +1) }{\Gamma(\alpha s + \alpha \beta +1) } \frac{\Gamma(\alpha s + \alpha \beta +1) }{\Gamma(\alpha \beta +1)} =\M_{{\bf{e}}}(s),
	\end{eqnarray}
	which, by Mellin inversion, translates into   the following factorization of Markov operators
	\[ \Lambda_{\alpha,\beta}\: \Lambda_{\eab}=\Lambda_{\e}.\] This together with an application of the Jensen inequality yields, for any $f\in \Lg$, that
	\begin{eqnarray*}
	||\Lambda_{\alpha,\beta}\:f||^2_{\eab}&=&\int_0^{\infty}\Lambda^2_{\alpha,\beta}f(x){\bf{e}}_{\alpha,\beta}(x)dx \\& \leq& \int_0^{\infty}\Lambda_{\alpha,\beta}\:f^2(x){\bf{e}}_{\alpha,\beta}(x)dx =\int_0^{\infty}f^2(x)\e(x)dx=||f||^2_{\e},
	\end{eqnarray*}
	which proves the contraction property. Finally, with $p_n(x)=x^n, n\in \N$, observing that
	\begin{equation}\label{eq:poly_kern}
	\Lambda_{\alpha,\beta}\:\:p_n(x)=\M_{\lambda_{\alpha,\beta}}(n)p_n(x)=\frac{\Gamma(n +1)\Gamma(\alpha \beta +1) }{\Gamma(\alpha n + \alpha \beta +1) }p_n(x),
	\end{equation} the completeness of the range of $\Lambda_{\alpha,\beta}\:$ follows from Lemma \ref{lem:dense_pol} since the polynomials are dense in $\lnua$.
	\end{proof}
\subsection{Several analytical extensions  of $\mathcal{R}_n$}
Our next result provides several representations of the functions $\mathcal{R}_n$, which we recall to be defined,   for any $n \in \N$, and $x>0$, by
\begin{equation}
\mathcal{R}_n(x)=\frac{(-1)^n}{n!\eab(x)} (x^{n}\eab(x))^{(n)}.
\end{equation}
\begin{proposition}\label{prop:representation}
	For any $n\in \N$, the following analytical extensions of the co-eigenfunctions  $\mathcal{R}_n$ holds.
	\begin{enumerate}
		\item For any $z \in \C_\pi=\{z \in \C;\: |\arg(z)|<\pi\}$,
		\begin{eqnarray} \label{eq:rep_pol}
		\mathcal{R}_n(z) &=& \sum\limits_{k=0}^n { n \choose k}  \frac{\Gamma(k+1)\Gamma(n+\beta+\frac{1}{\alpha})}{\Gamma(k+\beta+\frac{1}{\alpha})} b(k)z^\frac{k}{\alpha},
		\end{eqnarray}
		where $b(k) =   \sum_{j=1}^k B_{k,j}\lb \frac{\Gamma(1-\frac{1}{\alpha})}{\Gamma(-\frac{1}{\alpha})},\frac{\Gamma(2-\frac{1}{\alpha})}{\Gamma(-\frac{1}{\alpha})},\ldots,\frac{\Gamma(k-j+1-\frac{1}{\alpha})}{\Gamma(-\frac{1}{\alpha})} \rb$ and the $B_{k,j}'$s are the Bell polynomials.
		\item For any $z \in \C_\pi$,
		\begin{eqnarray} \label{eq:rn_w}
		\mathcal{R}_n(z)  	&=&  \frac{(-1)^n  }{n!} e^{z^{\frac1\alpha}} {}_1\Wr_1\left({\frac{1}{\alpha},  n+\beta+\frac1\alpha-1 \choose \frac{1}{\alpha}, \beta+\frac1\alpha-1} ;e^{i\pi} z^{\frac{1}{\alpha}}\right)
		\end{eqnarray}
		where ${}_1\Wr_1\left({\frac{1}{\alpha},  n+\beta+\frac1\alpha-1 \choose \frac{1}{\alpha}, \beta+\frac1\alpha-1} ;z\right) = \sum_{k=0}^{\infty}  \frac{\Gamma(k/\alpha+n+\beta+\frac1\alpha)}{\Gamma(k/\alpha+\beta+\frac1\alpha)} \frac{z^{k}}{k!}$ is the Wright hypergeometric function.
		\item  For any $z\in \C_\pi$,
		\begin{eqnarray}\label{eq:Hankel}
		\mathcal{R}_n(z) &=& \frac{(-1)^n  }{n!}e^{z^{\frac1\alpha}}\int_0^{\infty}e^{-r}\mathfrak{I}_{\alpha,\beta}(e^{i\pi}(rz)^{\frac{1}{\alpha}})r^{n+\beta+\frac1{\alpha}-1}dr,
		\end{eqnarray}
		where $\mathfrak{I}_{\alpha,\beta}(z) =  \sum_{k=0}^{\infty} \frac{1}{\Gamma(k/\alpha+\beta+\frac1\alpha)} \frac{z^{k}}{k!}$ is an entire function.
		\item Finally, for any $z \in \C_{\frac{\alpha \pi}{2}}$, and, any $a>1-\beta-\frac1\alpha$,
		\begin{equation}\label{eq:MellinInv_r}
		\mathcal{R}_n(z)=\frac{1}{\eab(z)}\frac{(-1)^n}{2\pi i  n!}\int_{a-i\infty}^{a+i\infty}z^{-s}\frac{\Gamma(s)}{\Gamma(s-n)}\Gamma\lbrb{\alpha s-\alpha+\alpha\beta+1}ds.
		\end{equation}
	\end{enumerate}
\end{proposition}
\begin{remark}
	\begin{enumerate}
		\item
		It is worth mentioning that each of the representation above plays a substantial role in the proof of the results. Indeed, for instance,  the polynomial type representation \eqref{eq:rep_pol} allows to derive easily that for each $n \in \N$, $\mathcal{R}_n \in \lnua$ as well as the completeness property of this sequence. The other representations  are used to derive different uniform asymptotic bounds  of the norm   $||\mathcal{R}_n||_{\eab}$.
		\item It is also interesting to note  that the several representations of $\mathcal{R}_n(z^{\alpha})$ lead to a polynomial.
	\end{enumerate}
\end{remark}
\begin{proof}
	Let us denote,  for any $n\in \N$ and $x\geq0$,
	\begin{equation} \label{eq:w_n}
	\wn(x)= \mathcal{R}_n(x) \eab(x)=\frac{(-1)^n}{n!\Gamma(\alpha \beta +1)} (x^{n+\beta_{\alpha}}e^{-x^{\frac1\alpha}})^{(n)},\end{equation}
	where  we have set $\beta_{\alpha}=\beta+\frac1\alpha-1$. On the one hand, we have
	\begin{equation*}
	(x^{n+\beta_{\alpha}}e^{-x^{\frac1\alpha}})^{(n)} =  x^{\beta_{\alpha}}\sum_{k=0}^n { n \choose k} \frac{\Gamma(n+\beta_{\alpha}+1)}{\Gamma(k+\beta_{\alpha}+1)} x^{k} (e^{-x^{\frac1\alpha}})^{(k)}.
	\end{equation*}
	Next, we recall that the Bell polynomials $B_{k,j}$ are defined by
	\[B_{k,j}(a_1,a_2,\ldots,a_{k-j+1})=\sum_{\tilde{l}_k=k;\bar{l}_k=l}\frac{k!}{j_1!j_2!\ldots j_{k-j+1}!}\lb\frac{a_1}{1!}\rb^{j_1}\lb\frac{a_2}{2!}\rb^{j_2} \ldots \lb\frac{a_{k-j+1}}{(k-j+1)!}\rb^{j_{k-j+1}},\] where $\tilde{l}_k=\sum_{j=1}^{k}jl_j$ and $\bar{l}_k=\sum_{j=1}^{k}l_j$. Then an application of Fa\`a di Bruno's formula yields
	\begin{eqnarray*}
		(e^{-x^{\frac1\alpha}})^{(k)} &=& e^{-x^{\frac1\alpha}} \sum_{j=1}^k(-1)^k B_{k,j}\lb \frac{x^{\frac{1}{\alpha}-1}}{\alpha},(\alpha-1)\frac{x^{\frac{1}{\alpha}-2}}{\alpha },\ldots,(-1)^{k-j+1}\frac{\Gamma(k-j+1-\frac{1}{\alpha})x^{\frac{1}{\alpha}+j-k-1}}{\Gamma(-\frac{1}{\alpha})} \rb\\
		&=& x^{\frac{k}{\alpha}-k}e^{-x^{\frac1\alpha}} \sum_{j=1}^k(-1)^k B_{k,j}\lb -\frac{\Gamma(1-\frac{1}{\alpha})}{\Gamma(-\frac{1}{\alpha})},\frac{\Gamma(2-\frac{1}{\alpha})}{\Gamma(-\frac{1}{\alpha})},\ldots,(-1)^{k-j+1}\frac{\Gamma(k-j+1-\frac{1}{\alpha})}{\Gamma(-\frac{1}{\alpha})} \rb,
	\end{eqnarray*}
	which provides the first representation as the analytical extension is obvious in this case.
	On the other hand, by expanding the exponential function in \eqref{eq:w_n} and differentiating term by term, which is allowed as the series defines an analytical function on the right-half plane,  we obtain
	\begin{eqnarray}
	\wn(x)&=& \frac{(-1)^n}{n!\Gamma(\alpha \beta +1)} (x^{n+\beta_{\alpha}}e^{-x^{\frac1\alpha}})^{(n)} \nonumber \\&=& \frac{(-1)^n}{n!\Gamma(\alpha \beta +1)}\sum_{k=0}^{\infty} (-1)^{k} \frac{\Gamma(k/\alpha+n+\beta_{\alpha}+1)}{\Gamma(k/\alpha+\beta_{\alpha}+1)} \frac{x^{\frac{k}{\alpha}+\beta_{\alpha}}}{k!} \label{eq:wn_wright}
	\end{eqnarray}
	from where we easily get the third representation.
	From the latter expression we get, by an  application of Fubini's theorem for analytical function, see  \cite[p.44]{Titchmarsh39},
	\begin{eqnarray}
	\wn(x)	&=& \frac{(-1)^n x^{\beta_{\alpha}}}{n!\Gamma(\alpha \beta +1)} \int_0^{\infty}e^{-r}\mathfrak{I}_{\alpha,\beta}(e^{i\pi}(rx)^{\frac{1}{\alpha}})r^{n+\beta_{\alpha}}dr.
	\end{eqnarray}
Finally, the Mellin-Barnes integral representation \eqref{eq:MellinInv_r} is obtained  from the expression \eqref{eq:Mellin_wn}, by the  Mellin inversion theorem which is justified, together with the analytical domain, from the estimates $|\M_{\wn}(a+ib-1)|\leq C_n |b|^{a-n-\frac{1}{2}}e^{-\alpha \frac{\pi}{2}|b|}$, valid for any $n \in \N$, $a>-\beta_\alpha$,  and $b$ large, see \cite{Paris01} for more details.
		\end{proof}

We end this part with the useful lemma.
\begin{lemma} \label{lem:dense_pol}
	The set of polynomials is dense in ${\rm{L}}^1(\eab)$, $\lnua$ and $\Le$. \end{lemma}
\begin{proof}
Since for  any $\alpha \in (0,1]$, $\beta\geq 1-\frac1\alpha$ and $0<a<1$, $\int_0^{\infty}e^{ax}\eab(x)dx<\infty$, we deduce that $\eab$ is moment determinate and hence by the Hahn-Banach theorem, one gets the first assertion. The last ones follow also by the moment determinacy of the measures combined with the  so-called Nevanlinna parametrization,  see \cite{Akhiezer-65}.	
	
	\end{proof}
\section{Proof of Theorems \ref{thm:main1}, \ref{thm:main2} and Proposition \ref{prop:sequence}} \label{sec:proof2}

The proof of these theorems will be split into several intermediate results. We start with the following, where $\co$, the space of continuous function on $\R_+$ vanishing at infinity is endowed with the uniform topology $||.||_{\infty}$ and  ${\rm{C}}^2_c(\R_+)$ stands for the space of twice continuously differentiable functions on $\R_+$ with compact supports.
\begin{lemma} \label{lem:sem}
There exists $\mathcal{D}_0$ a dense subset of $\co$ such that ${\rm{C}}^2_c(\R_+)\subset \mathcal{D}_0$ and	$(\mathbf{L}_{\alpha,\beta},\mathcal{D}_0)$	 is the generator of a Feller semigroup which is also denoted by $(P_t)_{t\geq0}$.
\end{lemma}
\begin{proof}
	First, note, from \eqref{eq:EK_deriv1},  that for any $f \in C^2_c(\R_+)$, we have
	\begin{eqnarray*}  \label{eq:EK_derive}
		\mathbf{L}_{\alpha,\beta}\:f(x) &=& \lb {\rm{d}}_{\alpha,\beta} -x\rb f'(x)+\frac{\sin(\alpha \pi)x}{\pi}\int_0^1 f''(xy) g_{\alpha,\beta}(y) dy  \\
		&=& \lb \frac{\Gamma( \alpha \beta+\alpha +1)}{\Gamma(\alpha \beta +1)} -x\rb f'(x)+\frac{x}{\Gamma(1-\alpha)}\int_0^1 f''(xy)\int_0^y \frac{r^{\beta+\frac1\alpha} }{(1-r^{\frac{1}{\alpha}})^{\alpha+1}} dr dy
	\end{eqnarray*}
	Then, since the mapping $r \mapsto \overline{g}_{\alpha,\beta}(r)=\frac{r^{\beta+\frac{1}{\alpha}} }{(1-r^{\frac{1}{\alpha}})^{\alpha+1}}$ is positive and non-increasing on $\lbrb{0,1}$ and satisfies $\int_0^1(1\wedge |\log r|) \overline{g}_{\alpha,\beta}(r)dr<\infty $, since
 $\lbrb{1-\lbrb{1-y}^{\frac1\alpha}}^{-\alpha-1}\stackrel{0}{\sim} \alpha^{-\alpha-1}y^{-\alpha-1}$ and $\log\lbrb{1-y}\sim -y$, according to \cite{Bartholme_Patie}, there exists $(K_t)_{t\geq0}$ a  $1$-selfsimilar Feller semigroup on $\R_+$, i.e.~for any $c>0$, $K_{ct}f(cx)=K_t f\circ d_c(x)$ with $d_cf(x)=f(cx)$, such  that $(\overline{\mathbf{L}}_{\alpha,\beta}f,\mathcal{D}_0)$,  with $\overline{\mathbf{L}}_{\alpha,\beta}f(x)= \mathbf{L}_{\alpha,\beta}\:f(x) + xf'(x)$, is its infinitesimal generator.  Next, let us define, for any $t\geq0$, $P_t f(x) = K_{e^t -1} f \circ d_{e^{-t}}(x)$, then  for each $t\geq0$, $P_t$ is plainly linear, with $P_t \co\subseteq \co$. Moreover, since by self-similarity $K_{e^t -1} f \circ d_{e^{-t}}(x) =K_{1-e^{-t}} f(e^{t}x)$, we get that $||P_tf||_{\infty}\leq ||f||_{\infty}$  and $\lim_{t\downarrow 0}P_tf= f$. Next, for any $t,s>0$,
	\begin{eqnarray*}
		P_t P_s f(x) &=& K_{1-e^{-t}}  K_{e^s-1} f \circ d_{e^{-s}}(xe^{t}) = K_{e^s-e^{-t}}  f \circ d_{e^{-s}}(xe^{t})\\  &=& K_{e^{t+s}-1}  f \circ d_{e^{-(t+s)}}(x)= P_{t+s}f(x).
	\end{eqnarray*}	
	Thus $(P_t)_{t\geq0}$ is also a  Feller semigroup. Next, for $f$ a smooth function, we have
	\begin{eqnarray*}
		\lim_{t \to 0 }\frac{P_tf(x)-f(x)}{t}
		&=& \lim_{t \to 0 }\frac{K_{1-e^{-t}}f(x)-f(x)+ K_{1-e^{-t}}f(e^{-t}x)-K_{1-e^{-t}}f(x)}{1-e^{-t}}  \\
		&=& \overline{\mathbf{L}}_{\alpha,\beta}f(x)-xf'(x) =\mathbf{L}_{\alpha,\beta}\:f(x),
	\end{eqnarray*}
	which completes the proof.
\end{proof}
Denoting by ${\rm{P}_n}$  the set of polynomials of order $n$, we have the following.
\begin{lemma} \label{lem:l_pol}
	For any $n \in \N$, $\mathbf{L}_{\alpha,\beta}:{\rm{P}_n}\rightarrow {\rm{P}_n}$. Moreover, $\eab(x)dx,x>0$, is an invariant  measure for the Feller semigroup $(P_t)_{t\geq0}$.  Consequently,  $(P_t)_{t\geq0}$ extends to a contraction semigroup in ${\rm{L}}^p(\eab)$, $1\leq p\leq\infty$.
\end{lemma}		
\begin{proof}
	First, observe that for any $k\geq 1$ (the case $k=0$ is obvious), writing $p_k(x)=x^k$, we have
	\begin{eqnarray} \label{eq:action_pol}
	\mathbf{L}_{\alpha,\beta}\:p_k(x) &=& \lb {\rm{d}}_{\alpha,\beta} -x\rb k \: p_{k-1}(x)+ \frac{\sin(\alpha \pi)}{\pi}k(k-1)p_{k-1}(x)\int_0^1 y^{k-2} g_{\alpha,\beta}(y) dy \nonumber \\  &=& k \Phi_{\alpha,\beta}(k) p_{k-1}(x) - kp_k(x),
	\end{eqnarray}
	where we used the relation \eqref{eq:poly}. By linearity, this proves the first claim.  As, for all $n\geq 0$, ${\rm{P}_n}\subseteq {\rm{L}}^1(\eab)$, we get, from Lemma \ref{lem:dense_pol}, that $\mathbf{L}_{\alpha,\beta}$ may be extended to a linear operator acting on ${\rm{Span}(P_n)}$, a dense subset of ${\rm{L}}^1(\eab)$. Next, for any $k \geq 1$ (the case $k=0$ is again obvious), we deduce, from \eqref{eq:action_pol}, the following
	\begin{eqnarray*} \label{eq:action_poln}
		\int_0^{\infty}\mathbf{L}_{\alpha,\beta}\:p_k(x)\eab(x)dx &=& k \Phi_{\alpha,\beta}(k)\int_0^{\infty}p_{k-1}(x)\eab(x) dx - k\int_0^{\infty}p_k(x)\eab(x)dx\\
		&=& k \Phi_{\alpha,\beta}(k)\frac{\Gamma(\alpha (k-1)+\alpha \beta +1)}{\Gamma(\alpha \beta +1)}- k\frac{\Gamma(\alpha k+\alpha \beta +1)}{\Gamma(\alpha \beta +1)}\\
		&=& k \frac{\Gamma(\alpha k+\alpha \beta +1)}{\Gamma(\alpha \beta +1)}- k\frac{\Gamma(\alpha k+\alpha \beta +1)}{\Gamma(\alpha \beta +1)}\\
		&=&0.
	\end{eqnarray*}
	Then,  by linearity and the discussion above,  we get that $\int_0^{\infty}\mathbf{L}_{\alpha,\beta}\:f(x)\eab(x)dx =0$, for
	any $f\in {\rm{Span}(P_n)}$ a dense subset of ${\rm{L}}^1(\eab)$,	which completes the statement about the existence of an invariant measure. A classical argument shows that the Feller semigroup $ (P_t)_{t\geq0}$ extends to a contraction semigroup in  ${\rm{L}}^1(\eab)$ and ${\rm{L}}^{\infty}(\eab)$, and, by Marcinkiewickz's interpolation Theorem to  ${\rm{L}}^p(\eab)$, $1\leq p\leq\infty$, see \cite{Stein}.
	\end{proof}

\begin{lemma}  \label{lem:eigen}
	Let $n\in \N$.  Then, $ \mathcal{P}_n \in \lnua$  and for any  $x\geq0$  we have
	\begin{equation} \label{eq:eigenf}
	\mathbf{L}_{\alpha,\beta}\:\mathcal{P}_n(x) = - n\mathcal{P}_n(x).
	\end{equation}
	Consequently, for any $n \in \N$ and $x\geq0$, we have
	\begin{equation} \label{eq:poly_inter}
	\Lambda_{\alpha,\beta}\: \mathcal{L}_n(x)=\mathcal{P}_n(x),
	\end{equation}
	and, the intertwining relation \eqref{eq:inter} holds.
\end{lemma}
\begin{remark}
We note that when $\alpha=1$, then $\mathcal{P}_n(x) = \mathcal{L}^{(\beta)}_n(x)$ yielding easily to the characterization of $\mathbf{L}_{\beta}=\mathbf{L}_{1,\beta}$ as the Laguerre differential operator.
\end{remark}
\begin{proof}
The first statement is obvious.	Next observe, from  \eqref{eq:action_pol}, that for any $n \in \N$, writing $\Gamma(\alpha \beta +1)\overline{\mathcal{P}}_n(x)=\mathcal{P}_n(x)$,  with $\mathcal{P}_n$ defined in \eqref{eq:Pn},
	\begin{eqnarray*}
		\mathbf{L}_{\alpha,\beta}\:\overline{\mathcal{P}}_n(x) &=& \sum_{k=0}^n (-1)^k\frac{{ n \choose k}}{\Gamma(\alpha k + \alpha \beta +1)} \mathbf{L}_{\alpha,\beta}\:p_{k}(x) \\
		&=& \lb -\sum_{k=0}^{n-1} (-1)^k\frac{{ n \choose k+1} (k+1) \Phi_{\alpha,\beta}(k+1)}{\Gamma(\alpha (k+1) + \alpha \beta +1)}  x^k - \sum_{k=1}^n (-1)^k\frac{{ n \choose k} k}{\Gamma(\alpha k + \alpha \beta +1)}    x^{k}\rb  \\
		&=& \frac{-n}{\Gamma(\alpha \beta +1)}+ \lb -\sum_{k=1}^{n-1} (-1)^k\frac{{ n \choose k+1}(k+1)+{ n \choose k} k}{\Gamma(\alpha k + \alpha \beta +1)}  x^k - \frac{(-1)^n n}{\Gamma(\alpha n + \alpha \beta +1)}    x^{n}\rb  \\
		&=& -n \overline{\mathcal{P}}_n(x),
	\end{eqnarray*}
	where we used \eqref{eq:def_phir} for the third equality,  and, for the last one the identity  ${ n \choose k+1}(k+1)+{ n \choose k} k = n { n \choose k}$. This proves \eqref{eq:eigenf}.
		The identity \eqref{eq:poly_inter} follows easily from the definition of the polynomials and the relation \eqref{eq:poly_kern}.  We deduce from \eqref{eq:eigenf} and the Cauchy problem \eqref{eq:cauchy},  that, for all $t\geq0$,
		\begin{equation} \label{eq:eig_pt}
		P_t \mathcal{P}_n(x) = e^{-nt} \mathcal{P}_n(x).
		\end{equation}
		Then, for any $n \in \N$ and $t\geq 0$, we get, from \eqref{eq:eig_pt},  that
		\begin{equation*}
		\Lambda_{\alpha,\beta}\: Q_t \mathcal{L}_n(x) = e^{-nt}\Lambda_{\alpha,\beta}\:\mathcal{L}_n(x) = P_t \Lambda_{\alpha,\beta}\: \mathcal{L}_n(x).
		\end{equation*}
		Since the operators are linear, we get that $\Lambda_{\alpha,\beta}\: Q_t f(x) =  P_t \Lambda_{\alpha,\beta}\: f(x)$, for all $f \in {\rm{Span}}(\mathcal{L}_n)$. By continuity of the involved operators in the appropriate Hilbert spaces, we get that the identity holds on $\overline{{\rm{Span}}}(\mathcal{L}_n)$ and hence from Lemma \ref{lem:dense_pol} on $\Lg$.
\end{proof}
As a consequence of the intertwining relation, we derive the following.
\begin{lemma} \label{lem:uniq}
	$\eab(x)dx,x>0$ is the unique invariant measure of the Feller semigroup $(P_t)_{t\geq0}$.
\end{lemma}	
\begin{proof}
Since $\co\subset \Le$, \eqref{eq:inter} holds also on $\co$. Next assume that there exists a measure  $\nu(dx) \neq \eab(x)dx$ such that for all $f\in \co$, $\nu P_t f = \nu f := \int_0^{\infty}f(x)\nu(dx)$. Since by dominated convergence, one has for any $f\in \co$, $\Lambda_{\alpha,\beta} f \in \co$, we get from the intertwining  relation, that
\[ \nu \Lambda_{\alpha,\beta} Q_t f = \nu  \Lambda_{\alpha,\beta} f, \]
that is $\bar{\nu}(x)dx=\int_{0}^{\infty}\lambda_{\alpha,\beta}(x/y)\frac{\nu(dy)}{y}dx$ is an invariant measure for  $(Q_t)_{t\geq0}$, and, thus by uniqueness of its invariant measure, we must have $\bar{\nu}(x)=e^{-x},\forall x>0$. This completes the proof by an appeal to a contradiction argument since from \eqref{eq:MellinIdentity} and  the multiplier $\mathcal{M}_{\lambda_{\alpha,\beta}}$ in \eqref{eq:def_Mellin_X} being zero free on $\Cb_{\lbrb{-1,\infty}}$, we get $\mathcal{M}_{\eab}=\mathcal{M}_\nu$ and thus $\nu(dx) = \eab(x)dx$.
	\end{proof}		
We say now that, for $n\in \N$,  $\mathcal{R}_n$ is a co-eigenfunction for $P_t$ associated to the eigenvalues $e^{-nt}$ if $\mathcal{R}_n \in \lnua$ and for any $f \in \lnua$,
\[ \spnu{P_tf,\mathcal{R}_n}{\eab}=e^{-nt}\spnu{f,\mathcal{R}_n}{\eab}.\]
We denote by $\Lambda^*_{\alpha,\beta}\:$ the adjoint of $\Lambda_{\alpha,\beta}\:$ in $\lnua$, i.e.~for any $f \in \lnua$ and $g\in \Le$, $\spnu{\Lambda^*_{\alpha,\beta}\:f,g}{\e}=\spnu{f,\Lambda_{\alpha,\beta}\:g}{\eab}$.
\begin{lemma}  \label{lem:co-eigen}
 For any $f \in \lnua$, we have for $a.e.~ x>0$,
\[ \e(x)\Lambda^*_{\alpha,\beta}\:f(x)=  \int_0^{\infty}f(xy)\eab(xy)\lambda_{\alpha,\beta}\:(1/y)\frac{dy}{y}. \]
Moreover, ${\rm{Ker}}(\Lambda^*_{\alpha,\beta})= \{\emptyset\}$ and for any $n\in \N$, the equation
\begin{equation} \label{eq:eq_rn}
\Lambda^*_{\alpha,\beta}\:f_n(x) = \mathcal{L}_n(x)
\end{equation}
has an unique solution in $\lnua$ given by $\mathcal{R}_n(x)=\frac{(-1)^n}{n!\eab(x)}(x^n \eab(x))^{(n)}$. Moreover, for all $n,t\geq0$, $\mathcal{R}_n$ is a co-eigenfunction of $P_t$ associated to the eigenvalue $e^{-nt}$ and $(P_t)_{t\geq 0}$ is non-self-adjoint. Finally, $\overline{{\rm{Ran}}}(\Lambda^*_{\alpha,\beta})= \Le$.
\end{lemma}
\begin{proof}
	Note that, for any $f,g \in \lnua$, $f,g \geq0$, we have that
	\begin{eqnarray*}
	\spnu{\Lambda_{\alpha,\beta}\:g,f}{\eab} &=& \int_0^{\infty} \int_0^{\infty}g(xy)\lambda_{\alpha,\beta}(y)f(x)\eab(x)dx\\ &=&\int_0^{\infty} \frac{g(r)}{\e(r)}\int_0^{\infty}\lambda_{\alpha,\beta}(r/x)f(x)\eab(x)\frac{dx}{x}\e(r)dr \\
	&=&  \int_0^{\infty} \frac{g(r)}{\e(r)}\int_0^{\infty}f(ry)\eab(ry)\lambda_{\alpha,\beta}(1/y)\frac{dy}{y}\e(r)dr \\
	&=&\spnu{\Lambda^*_{\alpha,\beta}\: f,g}{\e}.
	\end{eqnarray*}
Since if $f \in \lnua$, $|f| \in \lnua$, the first statement follows. Next, as, from Proposition \ref{prop:Kernel}, $\Lambda_{\alpha,\beta}\:$ has a dense range in $\lnua$, from a classical result on linear operators, ${\rm{Ker}}(\Lambda^*_{\alpha,\beta})= \{\emptyset\}$ from where we deduce that there exists  at most one solution of the equation \eqref{eq:eq_rn} in $\lnua$. Then, with the notation  ${\rm{R}}^{(n)} \e(x) = \frac{(-1)^n}{n!}(x^n \e(x))^{(n)}=\e(x) \mathcal{L}_n(x)$ and writing 	$\widehat{\Lambda}_{\alpha,\beta}f(x) = \int_0^{\infty}f(xy)\lambda_{\alpha,\beta}\:(1/y)dy/y =\int_0^{\infty}f(x/y)\lambda_{\alpha,\beta}\:(y)dy/y$, we see that if, for any $n\geq 0$, $\hat{f}_n$ is solution to the equation
\begin{equation}\label{eq:mellin_cov}
\e(x)\Lambda^*_{\alpha,\beta} f_n(x)=\widehat{\Lambda}_{\alpha,\beta}\hat{f}_n(x) = {\rm{R}}^{(n)} \e(x)
\end{equation}
and  $f_n = \frac{\hat{f}_n}{\eab} \in \lnua$ then, for $a.e.~x>0$, $f_n(x)$ is solution to \eqref{eq:eq_rn}.
Invoking  the theory of Mellin convolution in the distributional sense, as described in \cite[Chap.~11]{Misra-Lavoine}, since from the  Proposition \ref{prop:Kernel}, we have that $\lambda_{\alpha,\beta}$ defines clearly a distribution, then the equation \eqref{eq:mellin_cov} can be written, with the notation of \cite[Chap.~11.11]{Misra-Lavoine},  as
\[ \hat{f}_n \surd \lambda_{\alpha,\beta} (x) =  {\rm{R}}^{(n)} \e(x). \]
Taking the Mellin transform on both sides of this latter equation, one gets
\begin{equation*}
\M_{\hat{f}_n}(s-1)\M_{\lambda_{\alpha,\beta}}(s-1) = \frac{(-1)^n}{n!}\frac{\Gamma(s)}{\Gamma(s-n)} \Gamma(s),
\end{equation*}
where we have used for the Mellin transform of  ${\rm{R}}^{(n)} \e(x)$ the formula 11.7.7 in \cite{Misra-Lavoine}. That is, from \eqref{eq:def_Mellin_X},
\begin{equation} \label{eq:Mellin_wn}
\M_{\hat{f}_n}(s-1) = \frac{(-1)^n}{n!}\frac{\Gamma(s)}{\Gamma(s-n)} \Gamma\lb\alpha  s+\alpha \beta+1-\alpha \rb
\end{equation}
with the mapping $s\mapsto \M_{\hat{f}_n}(s-1)$  analytical in $\C_{(1-\beta-\frac{1}{\alpha}, \infty)}$. By means of \eqref{eq:est_gamma}, we have that for any $\epsilon>0$,  $|\M_{\hat{f}_n}(ib-1)|\leq  C_n e^{-(\alpha -\epsilon)\frac{\pi}{2}|b|}$ with $C_n=C(n)>0$. Thus, we deduce, from \cite[Theorem 11.10.1]{Misra-Lavoine} that the Mellin convolution \eqref{eq:mellin_cov} admits an unique solution in the sense of distribution, given, using the  formula aforementioned,  by
\[\hat{f}_n(x)=  {\rm{R}}^{(n)} \eab(x)=\frac{(-1)^n}{n!}(x^n \eab(x))^{(n)}.\]
Since the function $\eab \in\ci$, we have $\hat{f}_n \in \ci$. Moreover,
\[ f_n(x) =\frac{\hat{f}_n(x)}{\eab(x)} =\frac{(-1)^n}{n!\eab(x)}(x^n \eab(x))^{(n)}={\rm{R}}^{(n)}_{\eab} \eab(x)=\mathcal{R}_n(x),  \]
and, by Proposition \ref{prop:representation}\eqref{eq:rep_pol}, $f_n(x^\alpha)=\frac{\hat{f}_n(x^{\alpha})}{\eab(x^{\alpha})}$ is a polynomial and hence $f_n \in \lnua \cap \ci$. Thus, for all $x>0$, $\mathcal{R}_n(x)$
 is the unique solution in $\lnua$ of the equation \eqref{eq:eq_rn}, which completes the proof of the statement.
Next, we deduce from the previous identity and the fact that  $\overline{{\rm{Span}}}(\mathcal{L}_n)=\Le$, see Lemma \ref{lem:dense_pol}, that $\Lambda^*_{\alpha,\beta}\:$ has dense range in $\lnua$. Then, to prove that for each $n \in \N$, $\mathcal{R}_n$  is a co-eigenfunction, as the bounded operator $\Lambda_{\alpha,\beta}$ has a dense range in $\lnua$, it is enough to show that, for all $f \in \Le$,
\begin{eqnarray*}
 \spnu{P_t \Lambda_{\alpha,\beta}\: f,\mathcal{R}_n}{\eab}&=&e^{-nt}\spnu{\Lambda_{\alpha,\beta}\: f,\mathcal{R}_n}{\eab}.\end{eqnarray*}
Finally, by means of  the intertwining relation \eqref{eq:inter} and the fact that $Q_t$ is self-adjoint in $\lnub$ with the Laguerre polynomials $\lbrb{\mathcal{L}_n}_{n\geq 0}$ as eigenfunctions, we get that
\begin{eqnarray*}
\spnu{P_t \Lambda_{\alpha,\beta}\: f,\mathcal{R}_n}{\eab}&=& \spnu{\Lambda_{\alpha,\beta}\: Q_t  f,\mathcal{R}_n}{\eab} =\spnu{ Q_t  f,\Lambda^*_{\alpha,\beta}\:\mathcal{R}_n}{\e}= e^{-nt}\spnu{ f,\mathcal{L}_n}{\e}\\&=& e^{-nt}\spnu{ \Lambda_{\alpha,\beta}\: f,\mathcal{R}_n}{\eab},\end{eqnarray*}
which completes the proof, since for all $n\geq0$, $\mathcal{R}_n \neq \mathcal{P}_n$.
 \end{proof}
 \subsection{Proof of Proposition \ref{prop:sequence}}
The facts that $\mathcal{P}_n, \mathcal{R}_n \in \lnua$ for all $n\geq 0$ and $\overline{{\rm{Span}}}(\mathcal{P}_n)=\lnua$ follow easily from lemmas  \ref{lem:eigen}, \ref{lem:co-eigen} and \ref{lem:dense_pol}. Next, for any $f \in \lnua$, a simple change of variable yields that  %\int_0^{\infty} f^2(x)\frac{x^{\beta_{\alpha}}e^{-{x}^{\frac{1}{\alpha}}}}{\Gamma(\alpha \beta +1)}dx =\alpha  \int_0^{\infty} f^2(x^{\alpha})\frac{ x^{\alpha \beta_{\alpha}+\alpha-1}e^{-{x}}}{\Gamma(\alpha \beta +1)}dx
\[ ||f||^2_{\eab} = \frac{\alpha \Gamma(\alpha \beta_{\alpha}+\alpha)}{\Gamma(\alpha \beta +1)}||f \circ p_{\alpha}||^2_{\e_{ \tilde{\beta}_{\alpha}}} \]
with $\tilde{\beta}_{\alpha} = \alpha \beta_{\alpha}+\alpha-1$ and  $f\circ p_{\alpha}(x)=f(x^{\alpha}) \in L^2({\e_{ \tilde{\beta}_{\alpha}}})$. Thus, the two Hilbert spaces are isomorphic. Since the polynomials are dense in ${\rm{L}}^2({\e_{ \tilde{\beta}_{\alpha}}})$ and $\mathcal{R}_n \circ p_{\alpha}(x) =P_n^{\alpha}(x)$ is a polynomial of order $n$, we deduce from a standard result, see e.g.~\cite[Chap.~2.6]{Higgins}, that $\overline{{\rm{Span}}}(\mathcal{R}_n)=\lnua$. 
Next, using successively \eqref{eq:poly_inter} and \eqref{eq:eq_rn}, observe that for any $n,m \in \N$,
	\begin{equation*}
	\spnu{\mathcal{P}_n,\mathcal{R}_m}{\eab} =	\spnu{\Lambda\mathcal{L}_n,\mathcal{R}_m}{\eab} =	\spnu{\mathcal{L}_n,\Lambda^*\mathcal{R}_m}{\e}= \spnu{\mathcal{L}_n,\mathcal{L}_m}{\e}= \delta_{nm},
	\end{equation*}
	which shows that the sequences are both biorthogonal and  minimal. Next, by means of \eqref{eq:poly_inter} and the Parseval identity of the Laguerre polynomials, we get, for any $f \in \lnua$,
	\begin{equation*}
	\sum_{n=0}^{\infty}|\spnu{f,\mathcal{P}_n}{\eab}|^2 =\sum_{n=0}^{\infty}|\spnu{\Lambda_{\alpha,\beta}^*f,\mathcal{L}_n}{\e}|^2 = ||\Lambda_{\alpha,\beta}^*f||_\e   \leq ||f||_{\eab},
	\end{equation*}
 	which provides the Bessel property of $\Pns$. It remains to show that $\Pns$ is not a Riesz basis. By the open mapping theorem and $\Lambda_{\alpha,\beta}\mathcal{L}_n=\mathcal{P}_n$, $n\geq 0,$ it is enough to show that $\Lambda_{\alpha,\beta}$ is not bounded from below. Observing, from \eqref{eq:poly_kern}, that
 	\begin{eqnarray*} \frac{||\Lambda_{\alpha,\beta}\:p_n||_{\eab}}{||p_n||_\e}&=& \frac{\Gamma(n +1)\Gamma(\alpha \beta +1) }{\Gamma(\alpha n + \alpha \beta +1) }\frac{||p_n||_{\eab}}{||p_n||_\e}=\frac{\Gamma(n +1)\Gamma^{\frac12}(2\alpha n+\alpha \beta+1) }{\Gamma(\alpha n + \alpha \beta +1)\Gamma^{\frac12}(2n+1) }.  \end{eqnarray*}
 	As, by Stirling approximation,  $\frac{\Gamma(n +1)\Gamma^{\frac12}(2\alpha n+\alpha \beta+1) }{\Gamma(\alpha n + \alpha \beta +1)\Gamma^{\frac12}(2n+1) } \simi e^{-n(1-\alpha)\log 2}$, we get that
 	\begin{eqnarray*} \lim_{n \to \infty}\frac{||\Lambda_{\alpha,\beta}\:p_n||_{\eab}}{||p_n||_\e}&=& 0  \end{eqnarray*}
 	which completes the proof.
 	
\subsection{Proof of the Theorems \ref{thm:main1} \ref{thm:main2}}
 The proof of Theorem \ref{thm:main1} (resp.~\ref{thm:main2}) follows readily from the lemmas \ref{lem:l_pol}, \ref{lem:eigen}  and an application of the Hille-Yosida theorem, combined with the lemmas \eqref{lem:uniq} and \ref{lem:co-eigen} (resp.~Proposition \ref{prop:Kernel} and the lemmas \ref{lem:eigen} and \ref{lem:co-eigen}).

 \section{Proof of Proposition \ref{prop:crude_bound}} \label{sec:proofprop}

\subsection{The bound \eqref{eq:asympt_polyn}}
We start with the following observation.
	\begin{lemma} \label{lem:pol}
		The  sequence of polynomials $(\mathcal{P}_{n})$ are the Jensen polynomials associated to the generalized modified Bessel function    $\mathcal{I}_{\alpha,\beta}(x)=\Gamma(\alpha \beta +1)\sum_{n=0}^{\infty} \frac{1}{\Gamma(\alpha n + \alpha \beta +1) }\frac{x^n}{n!}$, i.e.~ for any $x, t \in \mathbb{R}$, we have
		\begin{equation} \label{eq: Jensen polynomials gen}
		e^{t} \mathcal{I}_{\alpha,\beta}(xt)= \sum^{\infty}_{n=0} \mathcal{P}_{n}(-x) \frac{t^{n}}{n!}.
		\end{equation}
		In particular, the sequence $\Pns$ is not orthogonal  in any weighted ${\rm{L}}^2$ space.		
	\end{lemma}
	\begin{proof}
		First, from \cite[Proposition 2.1(ii)]{Craven-Csordas-89_Jensen_Turan}, easy algebra yields the identity \eqref{eq: Jensen polynomials gen}. From an elegant result of Chihara \cite{Chihara-68}  stating that the Laguerre polynomials are the only sequence of orthogonal polynomials generating the so-called Brenke type function of the form \eqref{eq: Jensen polynomials gen}, we complete the proof.
	\end{proof}
Then, on the one hand, since, for any $p=0,\dots, n-1$, and $x \in \R$, where we modify here slightly  the notation to emphasize the dependency on the parameter $\beta$ in \eqref{eq:Pn},
		\begin{eqnarray}
		(\mathcal{P}^{\beta}_n(-x))^{(p)} &=& \sum_{k=p}^{n} \frac{\Gamma(k+1)}{\Gamma(k-p+1)}\frac{{ n \choose k}}{\Gamma(\alpha k + \alpha \beta +1)} x^{k-p} \nonumber \\&=&\frac{\Gamma(n+1)}{\Gamma(n-p+1)\Gamma(\alpha p + \alpha \beta +1)}\sum_{k=0}^{n-p} \frac{{ n-p \choose k}}{\Gamma(\alpha k   + \alpha (\beta+p) +1)} x^{k}\nonumber \\ &=& \frac{\Gamma(n+1)}{\Gamma(n-p+1)\Gamma(\alpha p + \alpha \beta +1)} \mathcal{P}^{\beta+ p}_{n-p}(-x). \label{eq:der_pol}
		\end{eqnarray}
		Next, from \eqref{eq: Jensen polynomials gen}, we get, after performing a change of variable,   that, for all $n,x> 0$,
		\begin{eqnarray} \label{eq:cont_p}
		\Pon(-x)&=&   \frac{n!}{2\pi i}x^{n}\oint_{nx} e^{z/x}\mathcal{I}_{\alpha,\beta}(z)\frac{dz}{z^{n+1}},
		\end{eqnarray}
		where the contour is a  circle centered at $0$ with radius $nx>0$. Since the series representation of $\mathcal{I}_{\alpha,\beta}$ defines an entire function, one obtains from the Stirling approximation,  that, for any $\beta\geq1- \frac1\alpha$, its order is $\limsup_{k\to \infty} \frac{k\ln k}{\ln (\Gamma(\alpha k+\gamma+1) k!)}=\frac{1}{\alpha+1}$ and its type is $\mathfrak{t}_{\alpha}=\frac{\alpha+1}{e}\limsup_{k \to \infty} k (\Gamma(\alpha k+\gamma+1) k!)^{-\frac{p}{k}}=(\alpha+1)\alpha^{-\ratio{\alpha}}$. Thus, a classical bound from its maximum modulus yields that for all $x>0$ and large $n$, $\max_{|z|=nx}|\mathcal{I}_{\alpha,\beta}(z)|\leq e^{\mathfrak{t}_{\alpha}(nx)^{\frac{1}{\alpha+1}}}$. Since plainly, for all $x>0$ and $n\in \N$, $	|\Pon(x)|\leq \Pon(-x)$, see \eqref{eq:Pn}, we deduce, from \eqref{eq:cont_p}, that
		\begin{eqnarray*} \label{eq:co_p}
		\nonumber|\Pon(x)|&\leq&  \frac{n!e^{\mathfrak{t}_{\alpha}(nx)^{\frac{1}{\alpha+1}}}}{2\pi i}x^n\oint_{nx} |e^{\frac{z}{x}}|\frac{dz}{|z|^{n+1}}
		= e^{\mathfrak{t}_{\alpha}(nx)^{\frac{1}{\alpha+1}}}\frac{n!e^{-n\ln n}}{2\pi }\int_{0}^{2\pi} e^{n\cos\theta}d\theta
		\leq C n^{\frac12} e^{\mathfrak{t}_{\alpha}(nx)^{\frac{1}{\alpha+1}}},
		\end{eqnarray*}
		where  for the last inequality we used the bound $n! \leq e^{1-n}n^{n-\frac12}$ and $\int_{0}^{2\pi} e^{n\cos\theta}d\theta<2\pi e^{n}$.
		Finally, from $\frac{\Gamma(n+1)}{\Gamma(n-p+1)}\simi n^p$, we prove \eqref{eq:asympt_polyn}, for any non-negative integer $p$, from \eqref{eq:der_pol}.
\subsection{The bound \eqref{eq:crude_bound}}
From \eqref{eq:wn_wright}, we get, by differentiating term by term,  that, for any $q \in \N$,
  \begin{eqnarray*}
  	\wn^{(q)}(x)&=&  \frac{(-1)^n}{\Gamma(\alpha \beta +1)}\sum_{k=0}^{\infty} (-1)^{k} \frac{\Gamma(k/\alpha+n_q+\beta^q_{\alpha}+1)}{\Gamma(k/\alpha+\beta^q_{\alpha}+1)} \frac{x^{\frac{k}{\alpha}+\beta^q_{\alpha}}}{n!k!}, \end{eqnarray*}
  	where we have denoted for brevity $\beta^q_\alpha=\beta_\alpha-q=\beta+\frac1\alpha-1-q$ and $n_q=n-q$.
Note that for any $m \in \N$ and  $0<r\leq m$ with $\bar{r}=[r]+1$, we have the immediate inequality
	\[ \frac{\Gamma(m+r+1)}{\Gamma(m+1)} \leq (m+\bar{r})\ldots (m+1)\leq \lbrb{m+\bar{r}}^{r+1}. \] % =m^{\bar{r}}(1+\frac{\bar{r}}{m})\ldots (1+\frac1m)\leq \lb \lb1+\frac{\bar{r}}{m}\rb m\rb^{r+1}. \]
	Let $n>2\max(q,-\alpha\beta^q_\alpha)=2q$
and put $K^q_{\alpha,n}=\frac{\alpha+1}{\alpha}+\frac{\labsrabs{\beta^q_{\alpha}}+1}{n}$. Then an application of the inequality above with the choice of $n$ gives
	\begin{eqnarray}
	\nonumber|\wn^{(q)}(x)| \hspace{-0.2cm} & \leq &  \frac{1}{\Gamma(\alpha \beta +1)}\lb \sum_{k=0}^{ n}  \frac{\Gamma(k/\alpha+n_q+\beta^q_{\alpha}+1)}{\Gamma(k/\alpha+\beta^q_{\alpha}+1)} \frac{x^{\frac{k}{\alpha}+\beta^q_{\alpha}}}{n!k!} +\sum_{k= n}^{\infty}  \frac{\Gamma(k/\alpha+n_q+\beta^q_{\alpha}+1)}{\Gamma(k/\alpha+\beta^q_{\alpha}+1)} \frac{x^{\frac{k}{\alpha}+\beta^q_{\alpha}}}{n!k!}\nonumber \rb \\\nonumber &\leq&  \frac{x^{\beta^q_\alpha+1}(K^q_{\alpha,n}n)^{\labs\beta^q_{\alpha}\rabs+2}}{\Gamma(\alpha \beta +1)} \sum_{k=0}^{n}  \frac{(K^q_{\alpha,n}nx)^{\frac{k}{\alpha}}}{\Gamma(k/\alpha+\beta^q_{\alpha}+1)k!} +\frac{e^{\alpha\beta_\alpha+o(1)}\lbrb{1+\gamma^q_n\alpha}^{n+1}}{\alpha^{n+1}\Gamma(\alpha \beta +1)}\sum_{k=n}^{\infty}  \frac{k^{n+1}}{n!} \frac{x^{\frac{k}{\alpha}+\beta^q_{\alpha}}}{k!} \nonumber \\
	&\leq&  \frac{x^{\beta^q_\alpha+1}(K^q_{\alpha,n}n)^{\labs\beta^q_{\alpha}\rabs+2}}{\Gamma(\alpha \beta +1)} \sum_{k=0}^{\infty}\frac{(K^q_{\alpha,n}nx)^{\frac{k}{\alpha}}}{\Gamma(k/\alpha+\beta^q_{\alpha}+1)k!}+\frac{e^{\alpha\beta^q_\alpha+o(1)}\lbrb{1+\gamma^q_n\alpha}^{n+1}}{\alpha^{n+1}\Gamma(\alpha \beta +1)}\sum_{k=n}^{\infty}  \frac{k^{n+1}}{n!} \frac{x^{\frac{k}{\alpha}+\beta^q_{\alpha}}}{k!},\label{eq:rest}
	\end{eqnarray}
	where $\gamma^q_n=\frac{1}{1+\frac{\alpha\beta^q_\alpha}{n}}=1+O\lbrb{\frac{1}{n}}$.
	Next, since the first series in the last inequality defines an entire function (at the argument $z^\alpha$), as in the proof of Lemma \ref{lem:pol}, we compute easily its order to be $\frac{\alpha}{\alpha+1}$ and its type to be $\mathfrak{t}_{\alpha}=(\alpha+1)\alpha^{-\ratio{\alpha}}$, to obtain that for large $n$,
	\begin{eqnarray*} \sum_{k=0}^{\infty}\frac{\lbrb{K^q_{\alpha,n}nx}^{\frac{k}{\alpha}}}{\Gamma(k/\alpha+\beta^q_{\alpha}+1)k!}  & =& O\left( e^{\mathfrak{t}_{\alpha}\lbrb{K^q_{\alpha,n}nx}^{\frac{1}{\alpha+1}}}\right).
	\end{eqnarray*}
	For the second term on the right-hand side of \eqref{eq:rest}, we have, from the Stirling approximation and recalling that $x\leq (\kappa n)^{\alpha}$, with  $\kappa<\ratio{\alpha}e^{-2}$, that
	\begin{eqnarray*}
		\frac{n^{-\alpha\beta^q_{\alpha}}(1+\gamma^q_n\alpha)^{n+1}}{\alpha^{n+1}}\sum_{k=n}^{\infty}  \frac{k^{n+1}}{n!} \frac{x^{\frac{k}{\alpha}+\beta^q_{\alpha}}}{k!} &\leq& \lbrb{1+\frac{\gamma^q_n}{\alpha}}^{n+1}\sum_{k=n}^{\infty}  \frac{k^{n+1} n^k}{n!} \frac{\kappa^k }{k!}\\
		&\leq & \lbrb{1+\frac{\gamma^q_n}{\alpha}}^n \sqrt{n}e^n \sum_{k=n}^{\infty} k^{\frac{3}{2}} e^{-(k-n)\ln(\frac{k}{n})}  (e\kappa)^k \\
		&\leq&
		n^{\frac{5}{2}}\lbrb{1+\frac{\gamma^q_n}{\alpha}}^{n+1} \frac{e^{2n}\kappa^n}{1-e\kappa} =\bo{1},
	\end{eqnarray*}
	where we used the bound $ n! \geq C n^{n-\frac12}e^{-n} $ and noted that $e^{(n-k)\ln(\frac{k}{n})} \leq 1,$ for $k\geq n$. This together with the fact that $\lim_{n \to \infty}K^q_{\alpha,n} = K_\alpha =\frac{1+\alpha}{\alpha}$ completes the proof.

 \section{Uniform bounds for $|\mathcal{W}_n(x)|$ via saddle points methods and proof of Proposition \ref{thm:bound}} \label{sec:bound}

 In this section we consider uniform bounds in $x$ and $n$ for $|\mathcal{R}_n(x)|$. We shall use two of its representations  as given in Proposition \ref{prop:representation} in order to obtain the best asymptotic bound for $||\mathcal{R}_n||_{\eab}$.
It will be more convenient to state our estimates in term of the function
 \[ \wn(x)= \mathcal{R}_n(x)\eab(x).\]
Note that according to our assumptions
 \[ \ra=\alpha\beta_\alpha=\alpha\beta-\alpha+1\geq 0.\]
 The next result collects all bounds we appeal to in our proofs.
 	\begin{proposition}\label{thm:Bounds}
 	We write, for any $\alpha \in (0,1)$,
 		\begin{eqnarray}\label{eq:calphaa} \label{eq:balphaa}\label{eq:Aalphaa}
 		\bca=\alpha^{\alpha}\frac{\cos^{\alpha+1}\lbrb{\frac{\pi}{2}\alpha}}{ \sin^{\alpha}\lbrb{\frac{\pi}{2}\alpha}}, \quad \bba=\alpha^{\alpha}\frac{\csc\lbrb{\frac{\pi}{2(1+\alpha)}}}{\sin^{\alpha}\lbrb{\frac{\pi\alpha}{2(1+\alpha)}}}\quad  \textrm{and} \quad
 		\baa=\lbrb{1+\alpha}^{\alpha+1},
 		\end{eqnarray}
 		with $\bca\leq\bba\leq \baa$. Then, we have the following bounds.
 		\begin{enumerate}
 			\item \label{eq:paris}For any $a>-\beta_\alpha$ and any fixed $x>0$, we have, for any $0<\epsilon<\alpha$ and $n$ large,
 			\begin{equation}\label{eq:EstimateTvNU}
 			|\wn(x)| =  {\rm{O}}\lb n^{\frac{3}{2}-a} \csc\lb( \alpha-\epsilon)\frac{\pi}{2}\rb^{n}x^{-a}\rb.
 			\end{equation}
			%and \textcolor{red}{to be moved to the proof}thus if we choose $a=-\frac{\beta_\alpha}{2}$ then for any  $\alpha>\epsilon>0$, $\bca>H>0$ and $0<x\leq Hn^\alpha$ we have with some constant $C=C(H,\epsilon)>0$ and $n\geq n_0(H,\epsilon)$
 		%	\begin{equation}\label{eq:smallX}
 		%	|\wn(x)|\leq  Cx^{\frac{\beta_\alpha}{2}}n^{\frac{3}{2}+\frac{\beta_\alpha}{2}}e^{n\lbrb{\ln\lb \csc\lbrb{\frac{(\alpha-\epsilon)\pi}{2}}\rb+\frac12 H^\frac{1}{\alpha}}}  e^{-\frac12x^{\frac1\alpha}}.
 		%	\end{equation}
 				\item\label{it:middleX}
 				Let  $x=\bar\kappa_\alpha(\theta_*)n^\alpha,\theta_*\in\lbrb{0,\frac{\pi}{2}}$,  where  \begin{equation}\label{eq:def_kappa}
 				\theta_* \mapsto \bar\kappa_\alpha(\theta_*)=\alpha^\alpha\lbrb{\frac{\sin\lbrb{(1+\alpha)\theta_*}}{\sin(\theta_*)}}\lbrb{\frac{\sin\lbrb{(1+\alpha)\theta_*}}{\sin\lbrb{\alpha \theta_*}}}^\alpha\end{equation}
 				is a decreasing function for each $\alpha\in\lbrbb{0,1}$ with $\bar\kappa_\alpha(0)=\baa,\bar\kappa_\alpha\lbrb{\frac{\pi}{2\lbrb{1+\alpha}}}=\bba,\bar\kappa_\alpha\lbrb{\frac{\pi}{2}}=\bca$. Then, for any $1>\bar{\alpha}>0$ and $\bar\theta\in\lbrb{0,\frac\pi2}$ uniformly on $\alpha>\bar{\alpha}$ and $0<\theta<\bar\theta$ we have  with $C\lbrb{\alpha,\theta}=C\lbrb{\bar\theta,\bar{\alpha}}\lb \cos\lbrb{\theta}\lbrb{\frac{\sin(\theta)}{\sin\lbrb{(1+\alpha)\theta}}}^{\frac{1}{\alpha}}\rb^{\ra+\frac{1}{2}},$ where $C\lbrb{\bar\theta,\bar{\alpha}}>0$, that
 				\begin{eqnarray}\label{eq:GeneralBound}
 					\quad\quad\quad\labsrabs{\wn(x)}
 					&\leq& C\lbrb{\alpha,\theta} x^{\beta_\alpha}e^{n\lbrb{-\alpha \frac{\sin\lbrb{(1+\alpha)\theta}\cos\lbrb{\theta}}{\sin\lbrb{\alpha\theta}}	+\ln\lbrb{\frac{\sin\lbrb{\theta}}{\sin\lbrb{\alpha \theta}}}}}.
 				\end{eqnarray}

 		 			\item\label{it:subOptX}
 			For any $1>\bar{\alpha}>0$ and $0<\epsilon<\bba-\bca$ there is $n_0\in\N_+$ such that for $n\geq n_0$ uniformly on $1\geq \alpha\geq \bar{\alpha}\,$ (resp. on $\alpha>0$) and $x\in\lbrb{\lbrb{\bca+\epsilon}n^\alpha,\baa n^\alpha}\,$(resp. on $x\in\lbrb{\bba n^\alpha,\baa n^\alpha}$)
 			 			\begin{eqnarray}\label{eq:middleX}
 			               \labsrabs{\wn\lbrb{x}}=\bo{x^{\beta_\alpha}e^{-\frac12 x^{\frac1\alpha}} e^{n\lbrb{-\ln\lbrb{\alpha}+\frac12\lbrb{1+\alpha}^{\frac{1+\alpha}{\alpha}}-1-\alpha}}}.
 			 			\end{eqnarray}
 			\item\label{it:LargeX} Uniformly on $ x \geq 	\baa n^\alpha$,  for any $\eta <1$,
 		\begin{eqnarray}\label{eq:MellinLargeY}
 			\labsrabs{\wn(x)}&=&\bo{\alpha^{-\frac{5}{2}}  \:x^{\beta_\alpha}e^{-\eta x^{\frac{1}{\alpha}}}e^{nH_{\alpha,\eta}}}
 		\end{eqnarray}
 		where, with $\lbbrb{\frac12,1} \subset E_{\alpha}=\lbbrb{(\alpha+1)^{-\frac{1}{\alpha}},1},\,\forall \alpha\in\lbrb{0,1}$,
 		\begin{eqnarray}\label{eq:Haeta}
 		\nonumber H_{\alpha,\eta}&=&\max\lbcurlyrbcurly{\eta\lbrb{1+\alpha}^{\frac{\alpha+1}{\alpha}}-(\alpha+1)-\ln(\alpha);\ln\lbrb{\eta^{-\alpha}-1}}\\
 		&=&\lb\eta\lbrb{1+\alpha}^{\frac{\alpha+1}{\alpha}}-(\alpha+1)-\ln(\alpha)\rb\ind{\eta  \notin E_{\alpha}}-\ln\lbrb{\eta^{-\alpha}-1}\ind{ \eta \in E_{\alpha}}
 		\end{eqnarray}
 		and $\lim_{\alpha\uparrow1}\lim_{\eta\to \frac12}H_{\alpha,\eta}=0$.
 				
 			\begin{comment}
 			content...
 			For any $\bar{\alpha}>0$, $h>0$ and $0<\epsilon<\bba-\bca$ there is $n_0\in\N_+$ such that for $n\geq n_0$ uniformly on $1\geq \alpha\geq \bar{\alpha}$ and $x\in\lbrb{\lbrb{\bca+\epsilon}n^\alpha,\bba n^\alpha}$
 			\begin{eqnarray}\label{eq:middleX}
               \labsrabs{\wn\lbrb{x}}=\bo{x^{\frac{\beta_\alpha}2}e^{-\frac12 x^{\frac1\alpha}} n^{\frac{\ra}{2}}e^{n\lbrb{-\ln\lbrb{\alpha}+\frac12\lbrb{1+\alpha}^{\frac{1+\alpha}{\alpha}}-1-\alpha}}}
 			\end{eqnarray}
 		\end{comment} 	
 		\begin{comment}	\item for any $\kappa>0$, and $C? 0<x\leq(\kappa \alpha)^{\alpha}n^{\alpha},$ and $\eta \leq \frac{1 + \ln(\kappa)}{K \kappa}$ for some $K\geq1$ and for any $\epsilon>0$, $H_{\alpha} = \sec\lb( 1-\alpha+\epsilon)\frac{\pi}{2}\rb$,
 			\item for any $\bca n^\alpha  \leq x \leq \bba n^\alpha $ and  $\alpha>\alpha_1>0$, we have $K_0 c^{-1}_{\alpha}<\eta<b_\alpha^{-1}$ and $H_{\alpha} = \ln(\alpha)$
 			\item for any $\bba n^\alpha \leq x \leq 	\baa n^\alpha$,   any $\alpha (\alpha+1)^{-\frac{\alpha+1}{\alpha}}\leq \eta<(1+\alpha)\overline{C}^{-\frac1\alpha}_\alpha=(1+\alpha)\alpha^{\frac1\alpha} \cos^{\frac1\alpha}\lbrb{\frac{\pi}{2(1+\alpha)}}\cos\lbrb{\frac{\pi\alpha}{2(1+\alpha)}}$ and $H_{\alpha} = \ln(\alpha)$.
 			 \end{comment}
 		\end{enumerate}
 	\end{proposition}
 	
To prove Proposition \ref{thm:Bounds} we resort to different saddle point approximations of the Mellin-Barnes representation \eqref{eq:MellinInv_r} of $\wn(x)$, that is for any $n\in \N$, $a>-\beta_\alpha$ and $x>0$,
 \begin{equation}\label{eq:MellinInv}
 \wn(x)=\frac{(-1)^n}{2\pi i\Gamma(n+1)}\int_{a-i\infty}^{a+i\infty}x^{-s}\frac{\Gamma(s)}{\Gamma(s-n)}\Gamma\lbrb{\alpha s+\ra}ds.
 \end{equation}\\
We discuss different scenarios: when $x$ is fixed and $n\to\infty$ and when $x$ belongs to different non-overlapping regions in $\mathbb{R}_+$ which vary with $n$. The latter is required by the estimates obtained via optimal application of a somewhat generalized saddle point method.
\subsection{The bound \eqref{eq:paris}}
From \eqref{eq:MellinInv}, we get that for fixed $x>0$ and $a$ as in the statement,
\[\labs \wn(x)\rabs\leq C x^{-a}\int_{-\infty}^{\infty} \frac{\labs\Gamma(a+ib)\rabs }{\Gamma(n+1)\labs\Gamma(a+ib-n)\rabs}|\Gamma\lbrb{\alpha a+\ra+i\alpha b}|db \]
where throughout $C$ stands for a generic positive constant.
The celebrated formula $|\Gamma(a+ib-n)|\labs\Gamma(n+1-a+ib)\rabs=\frac{\pi}{\labs \sin(\pi\lb a-n-ib\rb)\rabs}$, and  the uniform bound $\labs\sin(\pi (a+ib))\rabs\leq Ce^{\pi |b|}$ yield that
 	\begin{eqnarray}\label{eq:Inv_Mel_refl}
 		|\wn(x)| &\leq& Cx^{-a} \int_{\R}  \labs\frac{ \Gamma(a+ib)\Gamma\lb n-a+ib \rb}{\Gamma(n+1)} \Gamma\lbrb{\alpha (a+ib) + \ra}\rabs e^{\pi |b|}  db.
 	\end{eqnarray}
Next using the bound \eqref{eqn:RefinedGamma1}, we get, for any $0<\epsilon<\alpha$,
\begin{eqnarray*} % \label{eq:Inv_Mel_refl}
|\wn(x)|&\leq &  C x^{-a}\int_{-\infty}^{\infty}\frac{\labs\Gamma(n+1-a+ib)\rabs}{\Gamma(n+1)}e^{\lb 1-\alpha+\epsilon\rb\frac{\pi}{2}  |b|}db.
\end{eqnarray*}
Hence, using \cite[Lemma 2.6]{Paris01}, we obtain, for large $n$, that
\begin{eqnarray}
|\wn(x)|&\leq &  C x^{-a}n^{1-a} \frac{e^{n\ln n-n}}{\Gamma(n+1)}\sec\lb( 1-\alpha+\epsilon)\frac{\pi}{2}\rb^{n-\frac12}. \nonumber
\end{eqnarray}
The Stirling approximation, e.g. \eqref{eqn:RefinedGamma1}, for $\Gamma\lbrb{n+1}$ shows that \eqref{eq:EstimateTvNU} holds.
\subsection{The bounds \eqref{it:LargeX} and \eqref{it:middleX} of Proposition \ref{thm:Bounds}}

For sake of clarity, we present the proofs of our estimates by stating several intermediate results which emphasize the different key steps of the saddle point approach. We postpone the proofs of the ones requiring some technical developments to the next subsections. Throughout, we shall recall, assume and use the following relations
\[ \ra=\alpha\beta_\alpha=\alpha\beta-\alpha+1\quad \textrm{ and }     \quad  \overline{\varsigma}=1-\varsigma. \]
Note that according to our assumptions $\ra \geq 0$. We start with the following general upper bound which follows as a result of using different estimates of the gamma function.

\begin{lemma} \label{lem:bound1}
For any  $n\in \N$ and $\kappa>0$ with $n=\varsigma a$ and  $x=\lbrb{\kappa\alpha}^\alpha n^{\alpha}$, we have that on $0<\varsigma<\frac{\alpha n}{h}$, or equivalently $\alpha a>h$, for any $h>0$, that 
\begin{eqnarray}\label{eq:MellinInversion}
	\nonumber\labsrabs{\wn(x)}&\leq& C\alpha ^{\ra-\frac{1}{2}}n^{\ra}\varsigma^{-\ra-\frac{1}{2}}e^{n H_{\kappa}(\varsigma)}\IInf e^{ag_{\varsigma}(\tau)}\overline{R}_{\varsigma}(\tau)d\tau \\
	&=& C\alpha^{-\frac{1}{2}}( \kappa \varsigma)^{-\ra-\frac{1}{2}} \kappa^{\frac12} x^{\beta_{\alpha}}e^{n H_{\kappa}(\varsigma)}\IInf e^{ag_{\varsigma}(\tau)}\overline{R}_{\varsigma}(\tau)d\tau,
\end{eqnarray}	
where $C=C(h)>0$ is non-increasing in $h>0$,
\begin{equation}\label{eq:H}
H_{\kappa}(\varsigma)=-\lbrb{\frac{\alpha\ln(\kappa)}{\varsigma}+\frac{\alpha}{\varsigma}+\ln\varsigma+\frac{\alpha}{\varsigma}\ln\varsigma+\frac{\overline{\varsigma}}{\varsigma}\ln\labsrabs{\overline{\varsigma}}},
\end{equation}
 \begin{equation}\label{eq:g}
 g_{\varsigma}(\tau)=\frac{1}{2}\lbrb{(1+\alpha)\ln(1+\tau^2)-\overline{\varsigma}\ln\lbrb{1+\frac{\tau^2}{\overline{\varsigma}^2}}}-\tau\lbrb{(1+\alpha)\arctan{\tau}-\arctan\lbrb{\frac{\tau}{\overline{\varsigma}}}},
 \end{equation}
 and
  \begin{equation}\label{eq:R''}
 \overline{R}_{\varsigma}(\tau)=\lbrb{1+\tau^2}^{\frac{\ra-1}{2}}\lbrb{\overline{\varsigma}^2+\tau^2}^{\frac{1}{4}}.
  \end{equation}
\end{lemma}
 To optimize the upper bound of $|\wn(x)|$ we first investigate the function  $g_{\varsigma}(\tau)$ defined in \eqref{eq:g}. More precisely, we have that
 	\begin{equation}\label{eq:g'}
	g'_{\varsigma}(\tau)=-(1+\alpha)\arctan(\tau)+{\rm{sgn}}(\ov)\arctan\lbrb{\frac{\tau}{\labsrabs{\ov}}}+\pi\mathbb{I}_{\{\varsigma\geq 1\}},
 	\end{equation}
 	and, the following result. %\lbrb{\mathbb{I}_{\{\varsigma< 1\}}-\mathbb{I}_{\{\varsigma\geq 1\}}}
 \begin{lemma}\label{prop:funcG}
 	For all $\varsigma \geq 0$, $g_{\varsigma}(0)=0$. Moreover, the equation
 	\[ g'_{\varsigma}(\tau)=0\] has a non-zero solution $\tau_*=\tau(\varsigma)>0$ if and only if $\varsigma>\ratio{\alpha}$. Finally, for all $\varsigma \geq 0$, the mapping $\tau \mapsto g_{\varsigma}(\tau)$ attains a unique global maximum at $\tau_*=\tau(\varsigma)$, given by
 	\begin{equation}\label{eq:funcAtRoot}
 	g_{\varsigma}(\tau_*)=
 	 	\frac{1}{2}\lbrb{(1+\alpha)\ln(1+\tau_*^2)-\ov\ln\lbrb{1+\frac{\tau_*^2}{\ov^2}}}\mathbb{I}_{\{\varsigma>\frac{\alpha}{1+\alpha}\}}.
 	\end{equation}
 \end{lemma}
 Upon taking out $e^{ag_{\varsigma}(\tau_*)}$ in \eqref{eq:MellinInversion} and using that $n=a\varsigma$, we obtain
 \begin{eqnarray}\label{eq:MellinInversion1}
 \nonumber	\labsrabs{\wn(x)}& \leq& C\alpha ^{\ra-\frac{1}{2}}n^{\ra}\varsigma^{-\ra-\frac{1}{2}}e^{n\lbrb{ H_{\kappa}(\varsigma)+\frac{1}{\varsigma}g_{\varsigma}(\tau_*)}}I(a,\varsigma)\\
 	& \leq& C\alpha^{-\frac12}(\kappa \varsigma)^{-\ra-\frac{1}{2}} \kappa^{\frac12} x^{\beta_{\alpha}}e^{n\lb H_{\kappa}(\varsigma)+\frac{1}{\varsigma}g_{\varsigma}(\tau_*)\rb}I(a,\varsigma)
 \end{eqnarray}
 where the remainder integral expression is denoted by
 \begin{equation}\label{eq:I}
 I(a,\varsigma)=\IInf e^{a\lbrb{g_{\varsigma}(\tau)-g_{\varsigma}(\tau_*)}}\overline{R}_{\varsigma}(\tau)d\tau.
 \end{equation}
 We note that the saddle point method is not immediately applicable as the integrand in $I(a,\varsigma)$ depends on the parameter $\varsigma\in(0,\infty)$. In order to be able to estimate $I(a,\varsigma)$, we need to deliver  some additional  information on the mapping $\tau_*=\tau(\varsigma)$.  First we start though with a very useful lemma that will be used in the sequel. Note that %Denote by
 \begin{equation}\label{eq:rootEqn}
 g'_{\varsigma}\lbrb{\tau_*}=\left\{
 \begin{array}{ll}
 &-(1+\alpha)\arctan(\tau_*)+\arctan\lbrb{\frac{\tau_*}{1-\varsigma}}=0,\quad \varsigma<1\\
 &\pi-(1+\alpha)\arctan(\tau_*)-\arctan\lbrb{\frac{\tau_*}{\varsigma-1}}=0,\quad \varsigma\geq 1,
 \end{array}
 \right.
 \end{equation}
is simply \eqref{eq:g'} at the point of a unique global maximum $\tau_*\geq 0$. We have the following claim.
 \begin{lemma}\label{lem:wdagger} \label{lem:tau}
 	The solution of \eqref{eq:rootEqn} in terms of $\theta_*=\arctan\lbrb{\tau_*}\in\lbrb{0,\frac{\pi}{2}}$  is given by
 	\begin{equation}\label{eq:wdagger}
 	\varsigma(\theta_*)=1-\frac{\tan(\theta_*)}{\tan\lbrb{\lbrb{1+\alpha}\theta_*}}=\frac{\sin\lbrb{\alpha\theta_*}}{\sin\lbrb{(1+\alpha)\theta_*}\cos\lbrb{\theta_*}}.
 	\end{equation}
 	
 %	\begin{align}\label{eq:wdagger}
 %		\nonumber &\varsigma-1=-\frac{\tan(\theta_*)}{\tan\lbrb{\lbrb{1+\alpha}\theta_*}},\,\theta_*\in\lbrb{\frac{\pi}{2(1+\alpha)},\frac{\pi}{2}};\\ &1-\varsigma=\frac{\tan(\theta_*)}{\tan\lbrb{\lbrb{1+\alpha}\theta_*}},\,\theta_*\in\lbrb{0,\frac{\pi}{2(1+\alpha)}}.
 %	\end{align}
 	Moreover,  $\theta_*\mapsto \varsigma(\theta_*)$ is increasing on $(0,\frac{\pi}{2})$ with range $\lbrb{\ratio{\alpha},\infty}$% and $h^2(\varsigma):=\frac{\tau_*}{\ov^2}=\tan^2\lbrb{\lbrb{1+\alpha}\theta_*}$.
 	and the following holds.
 	\begin{enumerate}
 		\item\label{it:increasing}   $\varsigma\mapsto \tau_*(\varsigma)$  is non-decreasing on $\lbrb{0,\infty}$ with $\tau_*(\varsigma)=0$ on $(0,\ratio{\alpha}]$ and $\tau_*(1)=\tan\lbrb{\frac{\pi}{2(1+\alpha)}}$.
 		\item  $\varsigma \mapsto h\lbrb{\varsigma}:=\frac{\tau_*(\varsigma)}{\labsrabs{\ov}}=\tan\lbrb{\lbrb{1+\alpha}\theta_*}$  is increasing on $\lbrb{\frac{\alpha}{1+\alpha},1}$ and  decreasing on $(1,\infty)$, with $h\lbrb{\frac{\alpha}{1+\alpha}}=0$, $\lim_{\varsigma \to 1}h(\varsigma)=\infty$ and   $\lim_{\varsigma \to \infty}h(\varsigma)=\tan\lbrb{\frac{\pi}{2}\lbrb{1-\alpha}}$.
 		%\item \textcolor{red}{Do we remove this?M. We certainly do not need it anywmore I reckon} If furthermore $\alpha=1$ then $\tau^2(\varsigma)=\labsrabs{2\varsigma-1}=\labsrabs{\ov}\lbrb{2+\frac{1}{\labsrabs{\ov}}}$.
 	\end{enumerate}	
 \end{lemma}
We are now ready to provide  bounds for the remainder integral $I(a,\varsigma)$ in \eqref{eq:MellinInversion1}.

 \begin{lemma}\label{lem:IntegralTerm} \label{lem:improvedIntegralTerm}
 	\begin{enumerate}
 		\item  	For any $K\geq 1$, there exists $C=C(K)>0$ such that
 		\begin{eqnarray} \label{eq:IntegralTermSmall}
 		\sup_{\varsigma\leq K}I(a,\varsigma)&\leq& \frac{C}{\alpha^{2}}.\label{eq:IntegralTermLarge}
 		\end{eqnarray}
 		
 		\item
 		Fix $\bar{\alpha}>0$. Then  $\forall\alpha>\bar{\alpha}>0$ there exist uniform constants $K_0>2, C_0>0$ such that, for $\varsigma>K_0>2$,
 		\begin{equation}\label{eq:Iimproved}
 		I\lbrb{a,\varsigma}\leq C_0 \varsigma^{\ra+\frac{1}{2}}.
 		\end{equation}
 	\end{enumerate}
%and similarly
 	%\begin{equation}\label{eq:IntegralTermLarge}
 	%\sup_{K\geq \varsigma\geq 1}I(a,\varsigma)\leq \frac{C}{\alpha}.
 	%\end{equation}
 \end{lemma}

 Lemma \ref{lem:IntegralTerm} shows that for any choice of $\varsigma\leq K$ the upper bound in \eqref{eq:MellinInversion1} can be reduced to
 \begin{eqnarray}\label{eq:MellinInversion2}
 	\labsrabs{\wn(x)}&\leq
 	&C(K,\ra)\alpha^{-\frac52+\ra}n^{\ra}\varsigma^{-\ra-\frac{1}{2}}e^{nH^*_\kappa(\varsigma)}
	= C\alpha^{-\frac52}(\kappa \varsigma)^{-\ra-\frac{1}{2}}\kappa^{\frac12} x^{\beta_{\alpha}}e^{nH^*_\kappa(\varsigma)} \label{eq:boundk}
 		%&=& C\alpha^{-1}(\alpha\kappa \varsigma)^{-\ra-\frac{1}{2}}x^{\beta_{\alpha}}e^{n\mathcal{H}^*_{\kappa}(\varsigma)}e^{-\frac{x^{\frac{1}{\alpha}}}{\varsigma \kappa}},
 \end{eqnarray}
 where
 \begin{eqnarray}\label{eq:H0}
 H^*_\kappa(\varsigma)&=& H_\kappa(\varsigma)+\frac{1}{\varsigma}g_{\varsigma}(\tau_*)\mathbb{I}_{\{\varsigma>\ratio{\alpha}\}}.
 \end{eqnarray}

 For given $\kappa>0,n\geq 0$ we wish to minimize $H^*_\kappa(\varsigma)$. For this purpose, putting $\tau_*=\labsrabs{\overline{\varsigma}}h$, we observe, differentiating  \eqref{eq:rootEqn} with respect to $\varsigma$, that
 \begin{equation*}
 	\left\{
 	\begin{array}{ll}
 		&(1+\alpha)\frac{\tau'_*}{1+\tau^2_*}=\frac{h'}{1+h^2},\quad \varsigma<1,\\
 		-\!\!&(1+\alpha)\frac{\tau'_*}{1+\tau^2_*}=\frac{h'}{1+h^2},\quad \varsigma\geq 1.
 	\end{array}
 	\right.
 \end{equation*}
 Thus, from \eqref{eq:funcAtRoot} we are able to get that
 \[\frac{\partial }{\partial \varsigma} g_{\varsigma}(\tau_*)=\frac{1}{2}\ln\lbrb{1+\frac{\tau^2_*}{\ov^2}},\]
 and to conclude with the help of \eqref{eq:H} that
 \begin{eqnarray}\label{eq:H'}
 \frac{\partial}{\partial\varsigma} H^*_\kappa(\varsigma)&=&\frac{\alpha\ln(\kappa \varsigma)+\ln\labsrabs{\overline{\varsigma}}}{\varsigma^2}+\frac{\mathbb{I}_{\{\varsigma>\ratio{\alpha}\}}}{2\varsigma^2}\lbrb{\ln\lbrb{1+\frac{\tau^2_*}{\ov^2}}-\lbrb{1+\alpha}\ln\lbrb{1+\tau^2_*}}. %\\
%&=&  -\frac{1}{\varsigma}\lb H^*_\kappa(\varsigma)
 \end{eqnarray}
  As a result we have the following claim.
 \begin{lemma} \label{lem:H'}
 	The equation $\frac{\partial}{\partial\varsigma} H^*_\kappa(\varsigma)=0$, which is equivalent to
 	\begin{equation}\label{eq:lny}
 	-\alpha\ln(\kappa)=\alpha\ln(\varsigma)+\ln\labsrabs{\overline{\varsigma}}+\frac{1}{2}\lbrb{\ln\lbrb{1+\frac{\tau_*^2}{\ov^2}}-\lbrb{1+\alpha}\ln\lbrb{1+\tau_*^2}}\mathbb{I}_{\{\varsigma>\ratio{\alpha}\}},
 	\end{equation}
 	has a unique solution $\varsigma_*= \varsigma(\kappa)$  for all $\alpha \kappa>\bca^{\frac1\alpha}$. We have, with  $\kappa=\dg{\kappa}=\kappa(\varsigma_*)$ and $\overline{\varsigma}_*=1-\varsigma_*$,
 	\begin{equation}\label{eq:y}
 	\dg{\kappa}=\frac{1}{\varsigma_*}\lb\labsrabs{\overline{\varsigma}_*}^{-\frac{1}{\alpha}}\ind{\varsigma_*\leq\frac{\alpha}{1+\alpha}}+\lbrb{1+\tau^2_*}^{\frac{1}{2}}\lbrb{\frac{1+\tau_*^2}{\overline{\varsigma}_*^2+\tau_*^2}}^{\frac{1}{2\alpha}}\ind{\varsigma_*>\frac{\alpha}{1+\alpha}} \rb,
 	\end{equation}
 	which when $\varsigma_*>\ratio{\alpha}$, $\bar\kappa_\alpha(\theta_*)=(\alpha\kappa_{*})^{\alpha}$ is expressed by \eqref{eq:def_kappa} in terms of $\theta_*=\arctan(\tau_*)$, and, for all $\theta_* \in (0,\frac{\pi}{2})$,
 \begin{equation} \label{eq:bthe}	-\ln\varsigma_*+\ln\lbrb{\overline{\varsigma}_*}+\frac{1}{2}\ln\lbrb{1+\frac{\tau_*^2}{\lbrb{\dg{\varsigma}-1}^2}}=
 	\ln\lbrb{\frac{\sin\lbrb{\theta_*}}{\sin\lbrb{\alpha \theta_*}}}\leq -\ln(\alpha).\end{equation} Finally, if $\alpha \kappa<\bca^{\frac1\alpha}$, we have that $\frac{\partial}{\partial\varsigma} H^*_\kappa(\varsigma)<0$, for all $\varsigma>0$.
 	\end{lemma}
 	 \begin{proof}
 	 	The representation \eqref{eq:y}  of $\dg{\kappa}$ is immediate from the equation \eqref{eq:lny}, which in turn is the solution to $\frac{\partial}{\partial\varsigma} H^*_\kappa(\varsigma)=0$, see \eqref{eq:H'}. If $\varsigma> \ratio{\alpha}$, then \eqref{eq:def_kappa}  follows as a result of  the parametrization  $\theta_*=\arctan(\tau_*)$ in \eqref{eq:y}. Next from the fact that $\lim_{\kappa \to 0}\ln \kappa= -\infty$ we conclude that  $\frac{\partial}{\partial\varsigma} H^*_\kappa(\varsigma)<0$, for all $\varsigma>0$ and $\alpha\kappa<\bca$ since \eqref{eq:lny} has no solution. %Finally, \eqref{eq:GeneralBound} follows from  \eqref{eq:MellinInversion2} which is applicable due to the fact that when $0<\dg\theta<b<\frac{\pi}{2}$ we have that $\dg{\varsigma}\leq {\color{red}{What's that?}} K(b)<\infty$, see \eqref{eq:wdagger}. % %Let now $\varsigma\leq \ratio{\alpha}$. Then from \eqref{eq:y} we note that $\dg{\kappa}=\lb\varsigma_*\overline{\varsigma}_*^{\frac{1}{\alpha}}\rb^{-1}$, which is clearly decreasing in $\varsigma_*$ with $\dg{\kappa}\lbrb{\ratio{\alpha}}=\alpha^{-1}\lbrb{1+\alpha}^{\frac{\alpha+1}{\alpha}}=A_\alpha$. Then \eqref{eq:ywdagger} furnishes the rest of the claim as both ratios therein define non-increasing functions in $\lbrb{0,\frac{\pi}{2}}$.
 	 \end{proof}
\subsubsection{Proof of \eqref{eq:GeneralBound} and \eqref{eq:middleX}}
 When $\alpha\kappa>\bca^{\frac1\alpha}$ thanks to \eqref{eq:lny} of Lemma \ref{lem:H'} and upon substitution in \eqref{eq:H0} we have that
 \begin{equation}\label{eq:Hydagger}
 H^*_{\kappa}(\varsigma_*)=-\frac{\alpha}{\varsigma_*}-\ln\frac{\varsigma_*}{|\overline{\varsigma}_*|}+\frac{1}{2}\ln\lbrb{1+\frac{\tau^2_*}{\lbrb{\varsigma_*-1}^2}}\ind{\varsigma_*>\ratio{\alpha}}.
 \end{equation}

 Using \eqref{eq:Hydagger} with the parametrization $\dg\theta=\arctan\lbrb{\dg{\tau}}$, from \eqref{eq:MellinInversion2} and \eqref{eq:bthe},  we get \eqref{eq:GeneralBound}, since, in this setting, $\lbrb{\alpha \kappa_*}^\alpha=\bar{\kappa}_\alpha\lbrb{\theta_*}\in\lbrb{\bca,\baa}$ is represented by \eqref{eq:def_kappa}  and from Lemma \ref{lem:H'} $\varsigma_*>\ratio{\alpha}$. Next from \eqref{eq:MellinInversion2} taking out the term $-\frac{\alpha}{\varsigma_*}$ in \eqref{eq:Hydagger}, and using  $x=\bar\kappa_\alpha(\theta_*) n^\alpha$, we get
 	\begin{equation*}
 	|\wn(x)|\leq C \alpha^{\ra-\frac52}  \lb\bar{\kappa}_{\alpha}^{\frac1\alpha}(\theta_*) \varsigma_*\rb^{-\ra-\frac{1}{2}} \bar{\kappa}_{\alpha}^{\frac{1}{2\alpha}}(\theta_*) x^{\beta_{\alpha}}e^{-\frac12x^{\frac{1}{\alpha}}}e^{-n\ln(\alpha)} e^{ n\lb \frac12-\frac{\alpha \bar{\kappa}_{\alpha}^{-\frac1\alpha}(\theta_*)}{ \varsigma_*}\rb\bar{\kappa}_{\alpha}^{\frac1\alpha}(\theta_*)}.
 	\end{equation*}
 The estimate \eqref{eq:middleX} is obtained with the worst possible choices, that is,  $\bar{\kappa}_{\alpha}^{\frac1\alpha}(0)=\lbrb{1+\alpha}^{\frac{1+\alpha}{\alpha}}$, $\varsigma_*(0) = \ratio{\alpha}$ in the last exponent, $\bar{\kappa}_\alpha\lbrb{\theta_*}\in\lbrb{\bca,\baa}$ and  $\varsigma_*\geq \frac{\alpha}{1+\alpha}$.
%  \begin{lemma}\label{lem:LargeY}
 % 	Let $\dg{\varsigma}\leq\ratio{\alpha}$, i.e.~$\kappa\geq A_\alpha$. Then, for any $x=\lbrb{n\kappa\alpha}^\alpha$ we have that, for any $\eta<1$, with some $C=C\lbrb{\eta}>0$,
  %	\begin{eqnarray} %\label{eq:MellinLargeY}
  %	\nonumber \labsrabs{\wn(x)}&\leq&C\alpha^{-\ra-\frac{3}{2}}  e^{nH_{\alpha,\eta}}\:x^{\beta_{\alpha}}e^{-\eta x^{\frac{1}{\alpha}}}
  %	\end{eqnarray}
  %	where $H_{\alpha,\eta}=\lb\eta\lbrb{1+\alpha}^{\frac{\alpha+1}{\alpha}}-(\alpha+1)-\ln(\alpha)\rb\ind{\eta^{\alpha} \notin\lbrb{\frac{1}{\alpha+1},1};\eta<1}-\ln\lbrb{\eta^{-\alpha}-1}\ind{\frac{1}{\alpha+1}\leq \eta^{\alpha} \leq 1;\eta>0} $,\\ with $\lim_{\alpha \uparrow 1}\lim_{\eta\to\frac{1}{2}}H_{\alpha,\eta}=0$.
  %\end{lemma}
  \subsubsection{Proof of \eqref{eq:MellinLargeY}}
  	When $\varsigma\leq \ratio{\alpha}$, from \eqref{eq:y}, we have $\kappa_* \varsigma_* = \ov_*^{-\frac{1}{\alpha}}$ and thus %\eqref{eq:MellinInversion1} combined with \eqref{eq:Hydagger},
  	   \eqref{eq:MellinInversion2} with $K=1$  yields the following inequality, recalling that  $x^{\frac1\alpha}=\alpha\kappa_*n$,
  	\begin{eqnarray*}
  		\labsrabs{\wn(x)}&\leq& C\alpha^{-3}\ov_{*}^{\frac1\alpha\left(\ra+\frac{1}{2}\right)} e^{n	H^*_{\alpha,\eta}\lbrb{\dg{\varsigma}}}x^{\beta_{\alpha}+\frac{1}{2\alpha}} e^{-\eta x^{\frac{1}{\alpha}}},
  	\end{eqnarray*}
  	where we have set $H^*_{\alpha,\eta}(\dg{\varsigma})=-\frac{\alpha}{\varsigma_*}-\ln\lb \frac{\varsigma_*}{\overline{\varsigma}_*}\rb+\frac{\eta\alpha}{\varsigma_*\lbrb{\overline{\varsigma}_*}^{\frac{1}{\alpha}}}$ and  note that $\ov_{*}^{\frac1\alpha\left(\ra+\frac{1}{2}\right)}\leq 1$, for $0\leq \varsigma\leq \ratio{\alpha}$. Easy algebra gives then that
  	\[\frac{\partial }{\partial \dg{\varsigma}}H^*_{\alpha,\eta}(\dg{\varsigma})=\frac{\alpha-(1+\alpha)\dg{\varsigma}}{\dg{\varsigma}^2\dg{\overline{\varsigma}}}\lbrb{1-\frac{\eta}{\dg{\overline{\varsigma}}^{\frac{1}{\alpha}}}}.\]
  	Thus,  $H^*_{\alpha,\eta}(\dg{\varsigma})$ has at most one local maximum on $\lbrb{0,\ratio{\alpha}}$ either at  $\ratio{\alpha}$ or at $\dg{\overline{\varsigma}}^{\frac{1}{\alpha}}=\eta$, with $H^*_{\alpha,\eta}\lbrb{\frac{\alpha}{1+\alpha}}=\eta\lbrb{1+\alpha}^{\frac{1}{\alpha}+1}-(\alpha+1)-\ln(\alpha)$  and  $H^*_{\alpha,\eta}\lbrb{1-\eta^{\alpha}}=\ln\lbrb{\frac{\eta^{\alpha}}{1-\eta^{\alpha}}}$, which completes the proof of \eqref{eq:MellinLargeY}.

 \subsection{Proofs of the lemmas}
 \subsubsection{Proof of Lemma \ref{lem:bound1}}
Recall that $\ra=\alpha\beta_\alpha=1+\alpha\lbrb{\beta-1}$. First, from \cite[(2.2.30),\,Ch.2, p.50]{Paris01}
 since $\Re(s)>0$ we have, when $|\alpha s|\geq h$, that
 \[\labsrabs{\frac{\Gamma(\alpha s+\ra)}{\Gamma\lbrb{\alpha s}}}\leq C\labsrabs{\alpha s}^{\ra}\]
 for some constant $C=C(h)>0$.
 Hence, using this in \eqref{eq:MellinInv} with $\alpha a>h$, we get with some absolute constant $C>0$  that
 \begin{eqnarray}
 |\wn(x)|&\leq&\frac{C\alpha^{\ra}x^{-a}}{\Gamma(n+1)}\int_{-\infty}^{\infty}\labsrabs{a+ib}^{\bar{\beta}_\alpha}\labsrabs{\frac{\Gamma(a+ib)}{\Gamma(a-n+ib)}\Gamma\lbrb{a\alpha+i\alpha b}}db \nonumber \\
 &=&\frac{2C \alpha^{\ra}x^{-a}a^{\ra+1}}{\Gamma(n+1)}\IInf \labsrabs{1+i\tau}^{\bar{\beta}_\alpha}\labsrabs{\frac{\Gamma\lbrb{a\lbrb{1+i\tau}}\Gamma\lbrb{a\alpha\lbrb{1+i\tau}}}{\Gamma\lbrb{a\lbrb{\overline{\varsigma}+i\tau}}}}d\tau \label{eq:bound_1},
 \end{eqnarray}
 where we have performed the change of variables $\tau =\frac{b}{a}$ and $n=a\varsigma$.  We proceed with some estimates of the gamma functions. First,  from \cite[(2.1.8), Ch.2, p.70]{Paris01}, we get that for $\alpha a>h$
 \begin{eqnarray*}
 	\labsrabs{\Gamma\lbrb{a\lbrb{1+i\tau}}\Gamma\lbrb{\alpha a\lbrb{1+i\tau}}}
 	&\leq&C(h) e^{-a(1+\alpha)+a\ln(a)+a\alpha\ln(a\alpha)}\frac{e^{a\lbrb{\frac{(1+\alpha)}{2}\ln\lbrb{1+\tau^2}-(1+\alpha)\tau\arctan(\tau)}}}{a\alpha^{\frac{1}{2}}\lbrb{1+\tau^2}^{\frac{1}{2}}}.
 \end{eqnarray*}

 Next,  from  \cite[(2.1.8), Ch.2, p.70]{Paris01} using that $a\varsigma\ln\lbrb{a}=\ln\lbrb{\frac n\varsigma}^n$ with some absolute constant $C>0$, we have that
 \begin{eqnarray*}
 	\labsrabs{\Gamma\lbrb{a\lbrb{\overline{\varsigma}+i\tau}}}^{-1}&\leq& Ce^{-n+a}e^{-\lbrb{a-a\varsigma-\frac{1}{2}}\ln(a)}e^{-\lbrb{a-a\varsigma-\frac{1}{2}}\ln\labsrabs{\overline{\varsigma}+i\tau}+a\tau\arctan\frac{\tau}{\overline{\varsigma}}} \nonumber \\
 	&=&
 	Ce^{-n}n^{n}\varsigma^{-n}a^{\frac{1}{2}}e^{a-a\ln(a)}e^{a\lbrb{-\frac{\overline{\varsigma}}{2}\ln\lbrb{\overline{\varsigma}^2+\tau^2}+\tau\arctan\frac{\tau}{\overline{\varsigma}}}}\lbrb{\overline{\varsigma}^2+\tau^2}^{\frac{1}{4}}.
 \end{eqnarray*}
 Putting pieces together and using the Stirling formula for $\Gamma(n+1)$ yield with some $C=C(h)>0$, the inequality
 \begin{eqnarray} \label{eq:gn_est}
\labsrabs{1+i\tau}^{\bar{\beta}_\alpha}\frac{\labsrabs{\Gamma\lbrb{a\lbrb{1+i\tau}}\Gamma\lbrb{\alpha a\lbrb{1+i\tau}}}}{\Gamma(n+1)\labsrabs{\Gamma\lbrb{a\lbrb{\overline{\varsigma}+i\tau}}}}&\leq&
 C (\alpha a)^{-\frac{1}{2}}n^{-\frac{1}{2}} \varsigma^{-n} e^{-a\alpha+a\alpha\ln(a\alpha)} R_{a,\varsigma}(\tau)
 \end{eqnarray}
 where with $\overline{R}_{\varsigma}(\tau)=\lbrb{1+\tau^2}^{\frac{\ra-1}{2}}\lbrb{\overline{\varsigma}^2+\tau^2}^{\frac{1}{4}}$, see \eqref{eq:R''},
 \begin{eqnarray}\label{eq:R'}
 \nonumber R_{a,\varsigma}(\tau)&=&e^{\frac{a}{2}\lbrb{(1+\alpha)\ln\lbrb{1+\tau^2}-\overline{\varsigma}\ln\lbrb{\overline{\varsigma}^2+\tau^2}}}e^{-a\tau\lbrb{(1+\alpha)\arctan(\tau)-\arctan\frac{\tau}{\overline{\varsigma}}}}{\overline{R}}_{\varsigma}(\tau)\\& =& e^{a g_{\varsigma}(\tau)}e^{-a\overline{\varsigma}\ln\labsrabs{\overline{\varsigma}}}\overline{R}_{\varsigma}(\tau)=e^{ag_{\varsigma}(\tau)} e^{-n \frac{\overline{\varsigma}}{\varsigma}\ln\labsrabs{\overline{\varsigma}}}\overline{{R}}_{\varsigma}(\tau).
 \end{eqnarray}
 Plugging the upper bound \eqref{eq:gn_est} with the expression in \eqref{eq:R'} in \eqref{eq:bound_1}, we get that
 \begin{eqnarray*}
 	\labsrabs{\wn(x)}&\leq& C\alpha^{\ra-\frac12} n^{-\frac{1}{2}}a^{\ra+\frac{1}{2}}e^{-a \ln(x)-n\ln\varsigma-a\alpha+a\alpha\ln(a\alpha)-a\overline{\varsigma}\ln\labsrabs{\overline{\varsigma}}}\IInf e^{ag_{\varsigma}(\tau)}\overline{R}_{\varsigma}(\tau)d\tau,
 \end{eqnarray*}
which, after rearranging the terms by using the relations $x=(\alpha \kappa n)^{\alpha}$ and $n=\varsigma a$, completes the proof of the form \eqref{eq:MellinInversion}. The fact that $\varsigma\in(0,\frac{\alpha n}{h})$ follows from the fact that $a\alpha=\frac{n}{\varsigma}\alpha>h$.

\subsubsection{Proof of lemmas \ref{prop:funcG} and \ref{lem:wdagger}}
 First, for $\varsigma<1$, we note from \eqref{eq:g'} that
 $g''_{\varsigma}(\tau)=-\frac{1+\alpha}{1+\tau^2}+\frac{\ov}{\ov^2+\tau^2}<0,\,\forall\tau>0$
  	if and only if  $\tau^2>\frac{\ov}{\alpha+\varsigma}\lbrb{1- \ov(1+\alpha)},\,\forall\tau>0,$  which is equivalent to $\varsigma\leq \ratio{\alpha}$. Therefore, the mapping $\tau \mapsto g'_{\varsigma}(\tau)$ is decreasing  for $\varsigma\leq \ratio{\alpha}$ and otherwise increasing on a finite interval of the type $(0,b)$ and then decreasing to $\lim_{\tau \to \infty}g'_{\varsigma}(\tau)=-\frac{\pi}{2}\alpha<0$. Since  $g_{\varsigma}(0)=0$ we conclude the claim in this case.
  	For $\varsigma=1$ the claim is immediate, thus we assume in the sequel that $\varsigma>1$. It is clear that $\tau \mapsto g'_{\varsigma}(\tau)$ is decreasing, $g'_{\varsigma}(0)=\pi$  and $\lim_{\tau \to \infty}g'_{\varsigma}(\tau)=-\frac{\pi}{2}\alpha<0$, which completes the proof of the Lemma \ref{prop:funcG}.
  	Next, the proof of \eqref{eq:wdagger} is immediate from \eqref{eq:rootEqn}. The fact that $\varsigma=\varsigma(\theta_*)$ is increasing on $\lbrb{0,\frac{\pi}{2(1+\alpha)}}$ follows from the fact that both $\cos\lbrb{\dg{\theta}}$ and $\sin((1+\alpha)\theta_*)/\sin(\alpha \theta_*)$ are decreasing on this interval. The remaining portion between $\lbrb{\frac{\pi}{2(1+\alpha)},\frac{\pi}{2}}$ is dealt with \eqref{eq:wdagger} since $\tan(\theta_*)$ and $-\frac{1}{\tan\lbrb{\lbrb{1+\alpha}\theta_*}}$ are increasing on the interval. This completes the proof of the first part of Lemma \ref{lem:wdagger}.
  		Let $\varsigma<1$. Then we know from Lemma \ref{prop:funcG} that $\tau_*=0$, for $\varsigma\leq \ratio{\alpha}$. The fact that $\tau_*$ is increasing for $\varsigma\in\lbrb{\ratio{\alpha},\infty}$ follows from the properties reflected in \eqref{eq:wdagger} of Lemma \ref{lem:wdagger} which imply that $\theta_*=\arctan(\tau_*)$ and hence $\tau_*$ is increasing with $\varsigma$. The fact that $h(\varsigma)$ is increasing on $(0,1)$ and decreasing on $(1,\infty)$ follows from \eqref{eq:rootEqn} upon differentiation and using the fact that $\tau'_*\geq 0$. The values of   $\tau_*(1), \lim_{\varsigma \downarrow 1}h(\varsigma),\lim_{\varsigma \uparrow 1}h(\varsigma), \lim_{\varsigma \to \infty}h(\varsigma)$ follow from substitutions and manipulations of \eqref{eq:rootEqn}.
  		% Simple analysis yields that $g'_{\varsigma}(\tau)=0$ has a positive root iff $\varsigma>\frac{\alpha}{1+\alpha}$ and this corresponds to a maximum. Denote this root by $\tau_\circ=\tau(\varsigma)>0, \varsigma>\frac{\alpha}{1+\alpha}$ and $\tau_\circ=0,$ otherwise. We have then for $\varsigma>\frac{\alpha}{1+\alpha}$

  \subsubsection{Proof of Lemma \ref{lem:IntegralTerm}}

  		For $\varsigma\leq 1$ we have from \eqref{eq:R''} with $\overline{\rho}_{\alpha}=\frac{\ra}{2}-\frac{1}{4}$  that
  		\[\overline{R}_{\varsigma}(\tau)\leq \lbrb{1+\tau^2}^{\overline{\rho}_{\alpha}}.\]
  	Next according to Proposition \ref{prop:funcG} for each $\varsigma$, $g_{\varsigma}(\tau)$ attains a unique global maximum at $\tau_*$ and from Lemma \ref{lem:tau}\eqref{it:increasing} we have that $\sup_{\varsigma\leq 1}\tau_*(\varsigma)=\tau_*(1)=\tan\lbrb{\frac{\pi}{2(1+\alpha)}}$. Therefore, with $\frac{1}{1+\alpha}<\tilde\alpha=\frac{1+\frac{\alpha}{2}}{1+\alpha}<1$, we have  $\tan\lbrb{\frac{\pi}{2(1+\alpha)}}<\tan\lbrb{\frac{\pi}{2}\tilde{\alpha}}$ and we get that
  		\begin{eqnarray*}
  			I_1\lbrb{a,\varsigma}&=&\int_{0}^{\tan\lbrb{\frac{\pi}{2}\tilde{\alpha}}}e^{a\int_{\tau_*}^{\tau}g'_\varsigma(r)dr}
  			\overline{R}_{\varsigma}(\tau)d\tau\leq\int_{0}^{\tan\lbrb{\frac{\pi}{2}\tilde{\alpha}}}\lbrb{1+\tau^2}^{\overline{\rho}_{\alpha}}d\tau \leq K_1(\alpha),
  		\end{eqnarray*}
  where we set $K_1(\alpha) = \tan\lbrb{\frac{\pi}{2}\tilde{\alpha}}\lbrb{\frac{\mathbb{I}_{\lbcurlyrbcurly{\rra>0}}}{\cos^{2\overline{\rho}_{\alpha}}\lbrb{\frac{\pi}{2}\tilde{\alpha}}}+\mathbb{I}_{\lbcurlyrbcurly{\rra\leq 0}}}$.
  		However, we check that from \eqref{eq:g'} with $\varsigma\leq1$ that for $\tau>\tan\lbrb{\frac{\pi}{2}\tilde\alpha}$ we have that
  		\[g'_{\varsigma}(\tau)\leq \frac{\pi}{2}-(1+\alpha)\arctan\lbrb{\tan\lbrb{\frac{\pi}{2}\tilde\alpha}}\leq -\frac{\pi\alpha}{4}.\]
  		Thus,  for any $0<\alpha\leq 1$ and $\varsigma<1$ using $a=\frac{n}{\varsigma}>n$ and $\tau-\tau_*(\varsigma)\geq \tau-\tau_*(1)\geq \tau-\tan\lbrb{\frac{\pi}{2}\tilde\alpha}$ we have that
  		\begin{eqnarray*}
  			I_2\lbrb{a,\varsigma}&=&\int_{\tan\lbrb{\frac{\pi}{2}\tilde\alpha}}^\infty e^{a\int_{\tau_*}^{\tau}g'_\varsigma(r)dr}\overline{R}_{\varsigma}(\tau)d\tau\leq \int_{\tan\lbrb{\frac{\pi}{2}\tilde\alpha}}^{\infty}e^{-n\frac{\pi\alpha}{4}(\tau-\tan\lbrb{\frac{\pi}{2}\tilde\alpha})}\lbrb{1+\tau^2}^{\overline{\rho}_{\alpha}}d\tau \leq
  		K_2\lbrb{\alpha},
  		\end{eqnarray*}
  with $K_2\lbrb{\alpha}=\int_{0}^{\infty}e^{-\frac{\pi\alpha}{4}\tau}\lbrb{1+\lbrb{\tau+\tan\lbrb{\frac{\pi}{2}\tilde\alpha}}^2\mathbb{I}_{\lbcurlyrbcurly{\rra>0}}}^{\overline{\rho}_{\alpha}}d\tau$.
  		Thus, we have that
  		\[I(a,\varsigma)=I_1\lbrb{a,\varsigma}+I_2\lbrb{a,\varsigma}\leq K_1(\alpha)+K_2\lbrb{\alpha}.\]
  		However, as  $\lim_{\alpha \downarrow 0}\overline{\rho}_{\alpha}=-\frac{1}{4}+\lim_{\alpha \downarrow 0}\frac{\alpha\beta_\alpha}{2}=\frac14$ then clearly $\alpha^2K_1(\alpha)=o(1)$. Also immediately $\alpha^2 K_2(\alpha)=o(1)$. Thus \eqref{eq:IntegralTermSmall} follows. The proof of \eqref{eq:IntegralTermSmall} follows a similar pattern for $K\geq \varsigma>1$ with some $K>1$. To prove \eqref{eq:Iimproved}, choose, for any $\tilde{\varsigma}=\varsigma-1>1$, $\tau>\tilde{\varsigma}\tan\lbrb{\frac{\pi}{2}H}$, for some $0<H<1$. Then from \eqref{eq:g'} we have that
 	\[g'_{\varsigma}(\tau)\leq \pi-\frac{\pi}{2}H-(1+\bar\alpha)\arctan\lbrb{\tilde{\varsigma}\tan\lbrb{\frac{\pi}{2}H}}.\]
 	Therefore, for any $\frac{\bar{\alpha}}2>\epsilon>0$ small enough and $1>H_0>1-\frac{\bar{\alpha}}{2}$ there exists, $K_0=K_0(\epsilon)>2$ such that $\forall\alpha>\bar\alpha$ and $\varsigma>K_0$ we have that
 	$g'_{\varsigma}(\tau)\leq -\epsilon, \forall\tau>\tilde{\varsigma}T_0,$ with $T_0= \tan\lbrb{\frac{\pi}{2}H_0}$, and, thus we conclude that $\dg{\tau}<\tilde{\varsigma}T_0$. Using again that at $\dg{\tau}$ the function $g_{\varsigma}(\tau)$ attains a unique global maximum, we get using the expression for $\overline{R}_{\varsigma}(\tau)$ in \eqref{eq:R''} that
 	\begin{eqnarray*}
 		I_1(a,\varsigma)&=&\int_{0}^{\tilde{\varsigma}T_0}e^{a\int_{\tau_*}^{\tau}g'_{\varsigma}(r)dr}\:\overline{R}_{\varsigma}(\tau)d\tau\\ &\leq
 		&\int_{0}^{\tilde{\varsigma}T_0}\lbrb{1+\tau^2}^{\frac{\ra-1}{2}}\lbrb{\overline{\varsigma}^2+\tau^2}^{\frac{1}{4}}d\tau\\ &\leq
 		&\tilde{\varsigma}^{\frac{3}{2}}\int_{0}^{T_0}\lbrb{1+\tilde{\varsigma}^2\tau^2}^{\frac{\ra-1}{2}}\lbrb{1+\tau^2}^{\frac{1}{4}}d\tau.
 	\end{eqnarray*}
 	When $\ra\geq 1$ we then get by estimating at $\tau=T_0$ that $I_1(a,\varsigma)\leq C\varsigma^{\ra+\frac{1}{2}}$. Otherwise, if $0\leq \ra\leq 1$, we get by estimating only $\lbrb{1+\tau^2}^{\frac{1}{4}}\leq \lbrb{1+T_0^2}^{\frac{1}{4}}$ in the last integral and changing back variables that
 	\begin{align*}
 		&I_1(a,\varsigma)\leq  \tilde{\varsigma}^{\frac{1}{2}}\int_{0}^{\tilde{\varsigma}T_0}\lbrb{1+\tau^2}^{\frac{\ra-1}{2}}d\tau\leq C\tilde{\varsigma}^{\ra+\frac{1}{2}}\leq C\varsigma^{\ra+\frac{1}{2}}.
 	\end{align*}
 	Next, since $g'_{\varsigma}(\tau)\leq -\epsilon, \forall \tau>\tilde{\varsigma}T_0,$ and $\dg{\tau}<\tilde{\varsigma}T_0$, we have, recalling that $n=a\varsigma$,
 	\begin{eqnarray*}
 		I_2\lbrb{a,\varsigma}&=&\int_{\tilde{\varsigma}T_0}^{\infty}e^{a\int_{\tau_*}^{\tau}g'_{\varsigma}(r)dr}\: \overline{R}_{\varsigma}(\tau)d\tau\\
 		&\leq & \int_{\tilde{\varsigma}T_0}^{\infty}e^{-a\epsilon\lbrb{\tau-\tilde{\varsigma}T_0}}\lbrb{1+\tau^2}^{\frac{\ra-1}{2}}\lbrb{\overline{\varsigma}^2+\tau^2}^{\frac{1}{4}}d\tau \\ &=
 		&\int_{0}^{\infty}e^{-\frac{n}{\varsigma}\epsilon\tau}\lbrb{1+\lbrb{\tau+\tilde{\varsigma}T_0}^2}^{\frac{\ra-1}{2}}\lbrb{\overline{\varsigma}^2+\lbrb{\tau+\tilde{\varsigma}T_0}^2}^{\frac{1}{4}}d\tau.
 	\end{eqnarray*}
 	Since $\varsigma>K_0>2$ we have that $1/\varsigma<K^{-1}_0<1/2$, and, thus we get, by performing a change of variables,
 	\begin{equation*}
 		I_2\lbrb{a,\varsigma}\leq
 		\tilde{\varsigma}^{\frac{3}{2}} \int_{0}^{\infty}e^{-\lbrb{1-\frac{1}{K_0}}\epsilon\tau}\lbrb{1+\ov^2\lbrb{\tau+T_0}^2}^{\frac{\ra-1}{2}}\lbrb{1+\lbrb{\tau+T_0}^2}^{\frac{1}{4}}d\tau.
 	\end{equation*}
 	Again when $\ra\geq 1$ we get using that $\tilde{\varsigma}^2\lbrb{\tau+T_0}^2\geq T_0^2>0$ that for some $\tilde C=C(K_0,H_0)>0$ we have that
 	\[\lbrb{1+\tilde{\varsigma}^2\lbrb{\tau+T_0}^2}^{\frac{\ra-1}{2}}\leq \tilde{C}\tilde{\varsigma}^{\ra-1}\lbrb{\tau+T_0}^{\ra-1}\]
 	and we conclude, with some $C=C(K_0,H_0,\epsilon)>0$, that
 	\[I_2\lbrb{a,\varsigma}\leq C\varsigma^{\ra+\frac{1}{2}}.\]
 	Assume next that $0\leq \ra<1$. Clearly,
 	\[\sup_{\tau\geq 0}\lbrb{e^{-\frac{1}{2}\lbrb{1-\frac{1}{K_0}}\epsilon\tau}\lbrb{1+\lbrb{\tau+T_0}^2}^{\frac{1}{4}}}\leq C,\]
 	for some $C=C(H_0,K_0,\epsilon)>0$ and thus
 	\begin{eqnarray*}
 		I_2\lbrb{a,\varsigma} &\leq&
 		C\tilde{\varsigma}^{\frac{3}{2}} \int_{0}^{\infty}e^{-\frac{1}{2}\lbrb{1-\frac{1}{K_0}}\epsilon\tau}\lbrb{1+\tilde{\varsigma}^2\lbrb{\tau+T_0}^2}^{\frac{\ra-1}{2}}d\tau \\
 		&\leq& C T_0^{\ra-1}
 		\tilde{\varsigma}^{\ra+\frac{1}{2}}\int_{0}^{\infty}e^{-\frac{1}{2}\lbrb{1-\frac{1}{K_0}}\epsilon\tau}d\tau.
 	\end{eqnarray*}
 	Therefore, we again conclude that $I_2\lbrb{a,\varsigma}\leq C\varsigma^{\ra+\frac{1}{2}}$. Since $I\lbrb{a,\varsigma}=I_1\lbrb{a,\varsigma}+I_2\lbrb{a,\varsigma}$ we deduce the inequality \eqref{eq:Iimproved} and therefore conclude the proof of our lemma.

\subsection{Proof of Proposition \ref{thm:bound}}	
 	We are now ready to derive our upper bound for the  norms
 	of $\mathcal{R}_n$ in $\lnua$ and $\lnubb$.
 	
 	\subsection{The estimate \eqref{eq:bound_norm}}
 	Recall that $\ra=\alpha\beta_\alpha=1+\alpha\lbrb{\beta-1}$. Writing \[ F_n(x)=\Gamma(\alpha \beta +1) \wn^2(x) x^{-\beta_\alpha}e^{x^{\frac1\alpha}},\]
 	 we split the norms squared into three pieces as follows
 	 \begin{equation}
  ||\mathcal{R}_n||_{\eab}^2 =  \int^{\mathcal{B}_{\alpha} n^{\alpha}}_{0} F_n(x)dx + \int_{\mathcal{B}_{\alpha}n^{\alpha}}^{\overline{A}_{\alpha}n^{\alpha}} F_n(x)dx+\int_{\overline{A}_{\alpha}n^{\alpha}}^{\infty}F_n(x)dx,\end{equation}
  where we have set $\mathcal{B}_{\alpha}=2^{\alpha}\left(\frac{\baa^{\frac{1}{\alpha}}}{2} -\baa^{\frac{1}{\alpha+1}}\right)^{\alpha}= 2^{\alpha}\left(\frac{(\alpha+1)^{1+\frac{1}{\alpha}}}{2} -(\alpha+1)\right)^{\alpha}$. We have the following useful fact recalling that 	$\bca=\alpha^{\alpha}\frac{\cos^{\alpha+1}\lbrb{\frac{\pi}{2}\alpha}}{ \sin^{\alpha}\lbrb{\frac{\pi}{2}\alpha}}$ and $ \bba=\alpha^{\alpha}\frac{\csc\lbrb{\frac{\pi}{2(1+\alpha)}}}{\sin^{\alpha}\lbrb{\frac{\pi\alpha}{2(1+\alpha)}}}$.
  \begin{lemma} \label{lem:b}
  There exists $\bar{\alpha}>0$ such that, for any $0<\alpha<\bar{\alpha}$, $\mathcal{B}_{\alpha}>\bba$. Moreover, for any $\alpha \in (0,1),$ $\mathcal{B}_{\alpha}>\bca$.
  \end{lemma}
  \begin{proof}
  	The first claim follows from the inequality $\lim_{\alpha \to 0}\frac{\mathcal{B}^{\frac{1}{\alpha}}_{\alpha}}{2} =\frac{e-2}{2}>\frac1\pi=\lim_{\alpha \to 0}\frac{\bba^{\frac{1}{\alpha}}}{2}$. Next, we write
  	$f_1(\alpha)=\left(\alpha+1\right)^{\frac{\alpha+1}{\alpha}}$ and $\bca^{\frac1\alpha}=f_2(\alpha)f_3(\alpha)$, where $f_2(\alpha)= \frac{\alpha}{\sin \left(\frac{\pi }{2}\alpha\right)}$ and $f_3(\alpha)=\cos^{1+\frac{1}{\alpha}} \left(\frac{\pi }{2}\alpha\right)$. Note that $\mathcal{B}_{\alpha}>\bca$ is equivalent to  $f_1(\alpha)-f_2(\alpha)f_3(\alpha)>2\lbrb{\alpha+1}$. We have that $f_2$ is non-decreasing convex on $(0,1)$ and since  $f_1'(\alpha)=\frac{ \left(\alpha+1\right)^{\frac{\alpha+1}{\alpha}}}{\alpha^2}\left(\alpha -\ln(\alpha+1)\right)$ and $f_1''(\alpha)=\frac{f_1(\alpha) }{\alpha^4}\left(\ln^2(\alpha+1)-\frac{\alpha^2}{\alpha+1}\right)$, we deduce that $f_1$ is concave on $(0,1)$ with  $\lim_{\alpha \to 0} f_1(\alpha)-f_2(\alpha)=e-\frac{2}{\pi}>2$, and as $f_3(\alpha)\leq 1$, we have that there exists $\underline{\alpha}_1>0$ such that for any $0<\alpha <\underline{\alpha}_1$
  	\[f_1(\alpha)-f_2(\alpha)f_3(\alpha)\geq f_1(\alpha)-f_2(\alpha)> 2 \left(\alpha+1\right).\]
  	Repeating this argument, one constructs an increasing  sequence $(\underline{\alpha}_n)$, with $n\leq 10$, where $\underline{\alpha}_{n+1}$ is obtained from $\underline{\alpha}_{n}$  by using the bound for $\alpha \in [\underline{\alpha}_n,\frac12]$,  $f_3(\alpha)=\cos^{\frac{\alpha+1}{\alpha}}\left(\frac{\pi }{2}\alpha\right)\leq\cos^{3}\left(\frac{\pi }{2}\underline{\alpha}_n\right)<1$ yielding to the second claim in the case $\alpha\leq \frac12$. Now assume that $\alpha>\frac12$. Since $f_3(\alpha)\leq \bar{f}_3(\alpha) =\cos^2(\frac{\alpha \pi}{2})$ with $\bar{f}''_3(\alpha)=-\frac12 \pi^2 \cos(\pi \alpha)>0$ for $\alpha \in (\frac12,1]$  and $\lim_{\alpha \to 1} f_1'(\alpha)-f_3'(\alpha)=4(1-\ln(2))\approx 1.23 <2$ and as above one may construct a   sequence $(\underline{\alpha}_n)$ with $\underline{\alpha}_1=1$ and $\underline{\alpha}_4<\frac12$,  such that for any $\alpha \in (\frac12,\underline{\alpha}_n)$, $f_2(\alpha)\leq \frac{\underline{\alpha}_n}{\sin(\underline{\alpha}_n\frac{\pi}{2})}$ which guarantees the existence of $\underline{\alpha}_{n+1}<\underline{\alpha}_n$, such that for any $\alpha \in [\underline{\alpha}_{n+1},\underline{\alpha}_n] $,
  			\[\left(\alpha+1\right)^{\frac{\alpha+1}{\alpha}}-\frac{\underline{\alpha}_n}{\sin(\underline{\alpha}_n\frac{\pi}{2})}\bar{f}_3(\alpha)> 2(\alpha +1),\]
  	 	which completes the proof.
  	 	\end{proof}
  	 	  	 	
  	  	Recall that $\beta_{\alpha}=\beta+1-\frac{1}{\alpha}$. Then for each range we have the following estimates for large $n$.	
 	 \begin{enumerate}[1)] \item  On $ (0,\mathcal{B}_{\alpha}n^{\alpha}),$  we have from the estimate
 	 	  \eqref{eq:EstimateTvNU}, with $a=-\frac{\beta_{\alpha}}{2}$, $x=\kappa n^\alpha, \kappa<\mathcal{B}^{\frac1\alpha}_{\alpha}$ and any $0<\epsilon<\alpha$, that, for large $n$,
 	 	  \[ F_n(x)\leq C n^{\beta_{\alpha}+\alpha\beta_\alpha+3} e^{2n\ln\lb \csc\lbrb{\frac{(\alpha-\epsilon)\pi}{2}}\rb }  e^{x^{\frac1\alpha}} \leq C n^{\beta_{\alpha}+\alpha\beta_\alpha+3} e^{2n\lb\ln\lb \csc\lbrb{\frac{(\alpha-\epsilon)\pi}{2}}\rb + \frac{\mathcal{B}_{\alpha}^{\frac1\alpha}}{2}\rb } \]
 	 	  and hence
 	 	   \begin{eqnarray*}
 	 	   \int_{0}^{\mathcal{B}_{\alpha}n^{\alpha}} F_n(x)dx  &=& \bo{ n^{\beta_{\alpha}+\alpha\beta_\alpha+3+\alpha} e^{2n\left(-\ln\lb \sin\lbrb{\frac{(\alpha-\epsilon)\pi}{2}}\rb + \frac{1}{2}(\alpha+1)^{\frac{\alpha+1}{\alpha}}-(\alpha+1)\right)}}.\end{eqnarray*}
 	 \item On $ (\mathcal{B}_{\alpha} n^{\alpha},\overline{A}_{\alpha}n^{\alpha})$. Since, from Lemma \ref{lem:b}, $\mathcal{B}_{\alpha}>\bca$ and for small $\alpha$, $\mathcal{B}_{\alpha}>\bba$, we can use the estimates  \eqref{eq:middleX} to get
 	  \begin{eqnarray*}
 	  \int^{\overline{A}_{\alpha}n^{\alpha}}_{\mathcal{B}_{\alpha}n^{\alpha}} F_n(x)dx  &=& \bo{ n^{\alpha+\alpha\beta_\alpha}e^{2n\lb-\ln(\alpha)+\frac12(\alpha+1)^{\frac{\alpha+1}{\alpha}}-(\alpha+1)\rb}}.
 	  \end{eqnarray*}
% 	  where $C_{\alpha}= C \alpha^{-2} \lb\alpha(\alpha+1)^{\frac{1}{\alpha}}\rb^{-2\ra-1}\overline{A}_{\alpha}^{1+\beta_{\alpha}}.$
 	 	\item  On $ (\overline{A}_{\alpha}n^{\alpha}, \infty)$. Let $ \eta=\frac{1+\epsilon}2$ for any $0<\epsilon<\frac{1}{2}$  we have, from \eqref{eq:MellinLargeY},
 	 	 \[ F_n(x) \leq C_{\alpha}  e^{-2 n\ln\lbrb{(\frac12+\epsilon)^{-\alpha}-1}}\:x^{\beta_{\alpha}+\frac1\alpha}e^{-\epsilon x^{\frac{1}{\alpha}}} \]
 	 	 and thus, we get, for any $0<\varepsilon<\epsilon$,  that
 	 	\[ \int_{\overline{A}_{\alpha}n^{\alpha}}^{\infty}F_n(x)dx \leq C_{\alpha}   e^{-2 n\ln\lbrb{(\frac12+\epsilon)^{-\alpha}-1}}e^{-\varepsilon n \overline{A}^{\frac1\alpha}_{\alpha} } .\]
 	 Hence, %\label{eq:b3_o}
 	 	\begin{equation*}  \int_{\overline{A}_{\alpha}n^{\alpha}}^{\infty}F_n(x)dx = \bo{ e^{2 nT_{\alpha}}}.\end{equation*}
 	 \end{enumerate}
 	 To conclude the proof of the estimate \eqref{eq:bound_norm}, we note that, for all $\alpha \in (0,1)$,  \[T_{\alpha}=-\ln\lbrb{2^\alpha-1}\geq -\ln(\alpha)+\frac12(\alpha+1)^{\frac{\alpha+1}{\alpha}}-(\alpha+1)\geq -\ln(\sin\left(\frac{\alpha \pi}{2}\right))+\frac12(\alpha+1)^{\frac{\alpha+1}{\alpha}}-(\alpha+1), \] where the first inequality follows from \eqref{eq:Haeta} with $\eta=\frac12$ and the second from $\sin\lbrb{\frac{\alpha\pi}{2}}\leq\alpha,\forall\alpha\in\lbrb{0,1}$.
 \subsubsection{The estimate \eqref{eq:bound_norm_new}}

 We recall that for any $0<\gamma<\alpha$ and $\ew>0$ fixed
 \begin{equation*}
 	\ebb(x)=x^{\beta+\frac1\alpha-1} e^{\ew x^{\frac1\gamma}}, \: x>0,
 \end{equation*}
 and as above, writing,
 \[ \bar{F}_n(x)= \wn^2(x) x^{-\beta-\frac1\alpha+1}e^{-\ew x^{\frac1\gamma}},\]
 we split the norms squared into four pieces as follows
 \begin{equation*}
 	||\mathcal{R}_n||_{\ebb}^2 =  \int^{1}_{0} \bar{F}_n(x)dx + \int_{1}^{\overline{K}_{\alpha}n^{\alpha}}\bar{F}_n(x)dx+\int_{\overline{K}_{\alpha}n^{\alpha}}^{\baa n^{\alpha}}\bar{F}_n(x)dx+\int_{\baa n^{\alpha}}^{\infty}\bar{F}_n(x)dx,
 	\end{equation*}
 where  	  $\overline{K}_{\alpha}=e^{-2\alpha-\epsilon} \lbrb{\ratio{\alpha}}^{\alpha}$.
 \begin{enumerate}[1)]
 	\item On $(0,1)$ and $(1,\overline{K}_{\alpha} n^{\alpha})$, we use the bound \eqref{eq:crude_bound}, which yields that
 	\[ \bar{F}_n(x) \leq C n (nx)^{\beta+\frac1\alpha} e^{2\bar{\mathfrak{t}}_{\alpha}(nx)^{\frac{1}{\alpha+1}} -\ew x^{\frac1\gamma}}, \]
 	and thus, on the one hand,
 	\begin{eqnarray*}
 		\int^{1}_{0} \bar{F}_n(x)dx &\leq& C n^{1+\beta+\frac1\alpha} \int^{1}_{0}  x^{\beta+\frac1\alpha} e^{2\bar{\mathfrak{t}}_{\alpha}(nx)^{\frac{1}{\alpha+1}} -\ew x^{\frac1\gamma}} dx\\
 		&\leq &  C n^{1+\beta+\frac1\alpha} e^{2\bar{\mathfrak{t}}_{\alpha}n^{\frac{1}{\alpha+1}}}.
 	\end{eqnarray*}	
 	On the other hand,
 	\begin{eqnarray*}
 		\int^{\overline{K}_{\alpha} n^{\alpha}}_{1} \bar{F}_n(x)dx &\leq& C_{\alpha}n^{1+\beta+\frac1\alpha} \int^{K_{\alpha} n^{\alpha}}_{1}  x^{\beta+\frac1\alpha} e^{2\bar{\mathfrak{t}}_{\alpha}(nx)^{\frac{1}{\alpha+1}} -\ew x^{\frac1\gamma}} dx\\
 		&=& C_{\alpha}n^{1+\beta+\frac1\alpha+\alpha\beta +\alpha +1} \int^{\overline{K}_{\alpha}}_{n^{-\alpha}}  v^{\beta+\frac1\alpha} e^{n\lb 2\bar{\mathfrak{t}}_{\alpha}v^{\frac{1}{\alpha+1}} -\ew n^{\frac{\alpha}{\gamma}-1} v^{\frac1\gamma}\rb} dv.
 	\end{eqnarray*}
 	Next with $g_{n}(v)= 2\bar{\mathfrak{t}}_{\alpha}v^{\frac{1}{\alpha+1}} -\ew n^{\frac{\alpha}{\gamma}-1} v^{\frac1\gamma}$, we have $g'_{n}(v)=v^{\frac{1}{\alpha+1}-1}\lb \frac{2\bar{\mathfrak{t}}_{\alpha}}{\alpha+1} - \frac{\ew}{\gamma} n^{\frac{\alpha}{\gamma}-1} v^{\frac1\gamma-\frac{1}{\alpha+1}}\rb$ and simple computation yields that $g'_{n}(v) <0 $ on $(n^{-\alpha}, \overline{K}_{\alpha})$ and thus
 	\begin{eqnarray*}
 		\int^{\overline{K}_{\alpha} n^{\alpha}}_{1} \bar{F}_n(x)dx &\leq& C_{\alpha}n^{1+\beta+\frac1\alpha+\alpha}e^{2\bar{\mathfrak{t}}_{\alpha}n^{\frac{1}{\alpha+1}}}.
 	\end{eqnarray*}
 	Putting pieces together, we get
 	\begin{eqnarray}\label{eq:bn1}
 	\int^{K_{\alpha} n^{\alpha}}_{0} \bar{F}_n(x)dx &=& \bo{ n^{1+\beta+\frac1\alpha+\alpha}e^{2\bar{\mathfrak{t}}_{\alpha}n^{\frac{1}{\alpha+1}}}}.
 	\end{eqnarray}	
 	\item  On $ (\overline{K}_{\alpha} n^{\alpha},\baa n^{\alpha} )$. First, note that since $\alpha \mapsto \overline{K}_{\alpha}$ (resp.~$\alpha \mapsto  \baa$) is non-increasing (resp.~non decreasing) and $\lim_{\alpha \to 0} \overline{K}_{\alpha} = e^{-\epsilon} <\lim_{\alpha \to 0} \baa=1$, we have, $\overline{K}_{\alpha}<\baa$ for all $\alpha \in (0,1)$.  Then, using the bound  \eqref{eq:EstimateTvNU}, with $a=-\frac{\beta_{\alpha}}{2}$, and any $0<\epsilon<\alpha$, we get
 	\[ F_n(x)\leq C n^{\beta_{\alpha}+3} e^{2n\ln\lb \csc\lbrb{\frac{(\alpha-\epsilon)\pi}{2}}\rb }  e^{-\overline{K}^{\frac{1}{\gamma}}_{\alpha} n^{\frac{\alpha}{\gamma}}},  \]
 	and, hence
 	\begin{eqnarray} \label{eq:bn2}
 	\int_{\overline{K}_{\alpha} n^{\alpha}}^{\baa n^{\alpha}} \bar{F}_n(x)dx &=& \bo{ n^{\beta_{\alpha}+3+\alpha} e^{2n\ln\lb \csc\lbrb{\frac{(\alpha-\epsilon)\pi}{2}}\rb} e^{-\overline{K}^{\frac{1}{\gamma}}_{\alpha} n^{\frac{\alpha}{\gamma}}}}.\end{eqnarray}	
 	\item On $ (\baa n^{\alpha}, \infty)$, we use the bound \eqref{eq:MellinLargeY}, with $\eta=0$, to get
 	\begin{eqnarray}\label{eq:MellinInversion3}
 	\nonumber \bar{F}_n(x)
 	&\leq& C_{\alpha} x^{\beta_{\alpha}+\frac1\alpha}e^{2n(-\ln(\alpha)-(\alpha+1))}e^{-\ew x^{\frac1\gamma} },
 	\end{eqnarray}
 	and, hence for any $0<\varepsilon<\ew$, we have
 	\begin{eqnarray} \label{eq:bn3}
 	\int_{\baa n^{\alpha}}^{\infty} \bar{F}_n(x)dx &=&\bo{e^{-2n\ln(\alpha)-\varepsilon n^{\frac{\alpha}{\gamma}}}}.
 	\end{eqnarray}
 	The proof is completed by combining \eqref{eq:bn1}, \eqref{eq:bn2} and \eqref{eq:bn3}.
 \end{enumerate}

  \section{Proof of Theorem \ref{thm:main}} \label{sec:proof1}
  We have now all the ingredients to complete the proof of Theorem \ref{thm:main}. First, from the intertwining relation \eqref{eq:inter} and the expansion \eqref{eq:expansionLaguerre} of the Laguerre semigroup of order $\beta=0$,   we get, in the $\lnua$ topology, that for any $f \in \Le$,  $t >0$,
  \begin{eqnarray*}
  	P_t \Lambda_{\alpha,\beta}\: f(x) &=& \Lambda_{\alpha,\beta}\: Q_t f(x)  = \Lambda_{\alpha,\beta}\: \sum_{n=0}^{\infty}e^{-n t} \langle f,\mathcal{L}_n \rangle_{\e} \:  \: \mathcal{L}_n(x)
  	=  \sum_{n=0}^{\infty}e^{-n t} \langle f,\mathcal{L}_n \rangle_{\e} \:  \: \mathcal{P}_n(x),
  \end{eqnarray*}
  where the last identity is justified by the Bessel property of the sequence $(\mathcal{P}_n)$ combined with the fact that for any $f \in \Le$, the sequence $(\langle f,\mathcal{L}_n \rangle_{\e}) \in \ell^2(\N)$. Next since from Proposition \ref{prop:Kernel},  $\overline{{\rm{Ran }}}(\Lambda_{\alpha,\beta})\: = \lnua$ and ${\rm{Ker }}(\Lambda_{\alpha,\beta}) =\{\emptyset\}$, we have that $\Lambda^{-1}_{\alpha,\beta}$  is densely defined from $\lnua$ into $\Le$ and thus, for any $f \in {{\rm{Ran }}} (\Lambda_{\alpha,\beta})$ and $t>0$,
  \begin{eqnarray*}
  	P_t f(x) &=&  \sum_{n=0}^{\infty}e^{-n t} \langle \Lambda^{-1}_{\alpha,\beta}\: f,\mathcal{L}_n \rangle_{\e} \:  \: \mathcal{P}_n(x) \quad \textrm{ in } \lnua.
  \end{eqnarray*}
  Note that the two linear operators coincide on a dense subset of $\lnua$.  Using now \eqref{eq:eq_rn}, we deduce that, for any  $f \in {{\rm{Ran}}} (\Lambda_{\alpha,\beta})$ and  $t>0$,
  \begin{eqnarray} \label{eq:exp_dense}
  P_t f &=&  S_t f  \quad \textrm{ in } \lnua
  \end{eqnarray}
  where we have set
  \begin{equation} \label{eq:def_spe}
S_tf =\sum_{n=0}^{\infty}\left\langle P_t f,\mathcal{R}_n\right\rangle_{\eab}\mathcal{P}_n =\sum_{n=0}^{\infty}e^{-n t} \langle  f,\mathcal{R}_n \rangle_{\eab} \:  \: \mathcal{P}_n.
\end{equation}
From again the Bessel property  of the sequence $(\mathcal{P}_n)$, we have that the domain of $S_t$ is $\mathcal{D}(S_t) =\{ f \in \lnua; \: \left(e^{-n t} \langle  f,\mathcal{R}_n \rangle_{\eab}\right) \in \ell^2(\N)\}$.
Next, an application of the Cauchy-Schwarz inequality yields,  for any $f \in \lnua$, some $\epsilon>0$ and  $n$ large,
  \begin{eqnarray*}
  \left| \langle  f,\mathcal{R}_n \rangle_{\eab} \right| &\leq& ||f ||_{\eab}  ||\mathcal{R}_n ||_{\eab} \leq e^{nT_{\alpha}} ||f ||_{\eab}, \label{eq:est_b}
  \end{eqnarray*}
where we have used the bounds \eqref{eq:bound_norm}. Thus, for any $t>T_{\alpha}$, $ \mathcal{D}(S_t)=\lnua$ and using the synthesis operator as defined in \eqref{eq:def_synt}, we get that $S_t$  extends to a bounded linear operator in $\lnua$. Hence, from \eqref{eq:exp_dense}, since $\overline{\rm{Ran}}(\Lambda_{\alpha,\beta})=\lnua$, we conclude,  by an uniqueness argument, that for all $f\in \lnua$ and $t>T_{\alpha}$, $P_t =S_t$. Next, by means of the bound \eqref{eq:bound_norm_new}, we have for any $f \in \lnubb$, $n$ large and some $\epsilon>0$,
   \begin{eqnarray*}
  \left| \spnu{f,\mathcal{R}_n}{\eab } \right| &=& \spnu{f,\Rn\frac{\eab}{\ebb}}{\ebb}\leq ||f||_{\ebb} \left|\left|\Rn\frac{\eab}{\ebb}\right|\right|_{\ebb} \leq e^{(\bar{\mathfrak{t}}_{\alpha}+\epsilon)n^{\frac{1}{\alpha+1}}}||f||_{\ebb}. \label{eq:bound_cs}
  \end{eqnarray*}
Thus, plainly, for all $t>0$, $\lnubb \subseteq \mathcal{D}(S_t)$. Then,
as $(\mathcal{R}_n,\mathcal{P}_n)_{n\geq0}$ is a biorthogonal sequence, see Proposition \ref{prop:sequence}, we deduce from \eqref{eq:def_spe} that, for any $f \in \lnubb$, $t>0$ and  $m\geq0$,
\begin{equation*} \label{eq:sorth}
\left\langle S_t f,\mathcal{R}_m \right\rangle_{\nu} =\left\langle P_t f,\mathcal{R}_m\right\rangle_{\nu},
\end{equation*}
that is  $ S_t f - P_t f \in {\rm{Span}}(\mathcal{R}_n)^{\perp}$.  However, since from Proposition \ref{prop:sequence}, $\overline{{\rm{Span}}}(\mathcal{R}_n)=\lnua$, we conclude that $S_t f= P_t f$ in $\lnua$, as in a Hilbert space the notion of complete and total sequence coincide. Next using the bound \eqref{eq:asympt_polyn}, we get that, for any $p\in \N$ and $0\leq x<K$, for any $K>0$,
  \begin{eqnarray*} \label{eq:exp_deriva}
  	\left| \sum_{n=p}^{\infty}e^{-n t} \langle  f,\mathcal{R}_n \rangle_{\eab} \:  \: \mathcal{P}^{(p)}_{n}(x) \right|  &\leq & \sum_{n=p}^{\infty}e^{-n t} |\langle  f,\mathcal{R}_n \rangle_{\eab}| \: |\mathcal{P}^{(p)}_{n}(x)| \\
  	&\leq &C \sum_{n=p}^{\infty}e^{-n t} |\langle  f,\mathcal{R}_n \rangle_{\eab}| \:  n^{p+\frac12} e^{\mathfrak{t}_{\alpha}(nx)^{\frac{1}{1+\alpha}}}\\
  	&\leq & C \sum_{n=p}^{\infty}e^{-n t +\mathfrak{t}_{\alpha}(nK)^{\frac{1}{1+\alpha}}} |\langle  f,\mathcal{R}_n \rangle_{\eab}| \:  n^{p+\frac12},
  \end{eqnarray*} 	
  where,  from the preceding discussion, the last term is finite  whenever $t>0$ and $f  \in {{\rm{Ran}}}(\Lambda_{\alpha,\beta}) \cup \lnubb $ or $t>T_{\alpha}$ and $f \in \lnua$. Similarly, using in addition the bound \eqref{eq:crude_bound}, we get for any   integer $q$,  $0\leq x<K, 0<y<\bar{K}$, $K,\bar{K}>0$,
  \begin{eqnarray} \label{eq:ex_deriva}
  	\left| \sum_{n=p}^{\infty}e^{-n t} \mathcal{W}^{(q)}_n (y) \:  \: \mathcal{P}^{(p)}_{n}(x) \right|
  	&\leq & C y^{\beta + \frac{1}{\alpha}-q}\sum_{n=p}^{\infty}e^{-nt}n^{p+\frac12+\labs\beta + \frac{1}{\alpha}-1-q\rabs+2} e^{\widehat{K}_{\alpha}n^{\frac{1}{\alpha+1}}},
  \end{eqnarray}
  where $\widehat{K}_{\alpha}= \bar{\mathfrak{t}}_{\alpha}\bar{K}^{\frac{1}{\alpha+1}}   +\mathfrak{t}_{\alpha}K^{\frac{1}{1+\alpha}}$.
  Hence, we conclude easily that, for any $p,k \in \N$ and for such $t$ and $f$,
  \begin{eqnarray*} \label{eq:exp_derv}
  	\frac{d^k}{dt^k}(P_t f)^{(p)}(x) =\sum_{n=p}^{\infty}(-n)^k e^{-n t} \langle  f,\mathcal{R}_n \rangle_{\eab} \:  \: \mathcal{P}^{(p)}_{n}(x) ,
  \end{eqnarray*} 	
  where the series is locally uniformly convergent on $\R^+$.
  \begin{comment}From \eqref{eq:wn_wright}, we get, by differentiating term by term,  that, for any $q \in \N$,
  	\begin{eqnarray*}
  		\wn^{(q)}(y)&=&  \frac{(-1)^n}{n!\Gamma(\alpha \beta +1)}\sum_{k=0}^{\infty} (-1)^{k} \frac{\Gamma(k/\alpha+n+\beta_{\alpha}+1)}{\Gamma(k/\alpha+\beta_{\alpha}+1-q)} \frac{y^{\frac{k}{\alpha}+\beta_{\alpha}-q}}{k!} \\
  		&=& \frac{(-1)^q(n+q)!\Gamma(\alpha (\beta-q)+1 )}{n!\Gamma(\alpha \beta +1)}\mathcal{W}^{\beta-q}_{n+q}(y) \end{eqnarray*}
  	where we used the obvious notation. \end{comment}
  This combined with \eqref{eq:crude_bound}, which is uniform  in $y\in\lbrb{a,b}$, for large $n$, for any fixed couple $0<a<b<\infty$, gives that
  % \[ |\wn^{(q)}(y)|\leq C e^{(\bar{\mathfrak{t}}_{\alpha}+\epsilon)n^{\frac{1}{\alpha+1}}} y^{\beta_{\alpha}}e^{\ew y^{\frac{1}{\gamma}}}\]	and thus
    \begin{eqnarray*} \label{eq:exp_dervall}
  \frac{d^k}{dt^k}P^{(p,q)}_t(x,y) =\sum_{n=p}^{\infty}(-n)^k e^{-n t} \mathcal{W}^{(q)}_n(y)   \: \mathcal{P}^{(p)}_{n}(x),
  \end{eqnarray*}
  where the series is locally uniformly convergent on $\R^3_+$. Finally, observe that on the one hand, for all $t\geq0$,  $P_t\mathcal{P}_0(x) =1$, and hence according to \cite{Schilling_pos}, $(P_t)_{t\geq0}$ is a $C_b$-Feller semigroup. On the other hand, from \eqref{eq:ex_deriva}, we get that for all $t>0$, $(x,y)\mapsto P_t(x,y)$  is locally bounded and the strong Feller property follows from \cite[Corollary 2.2]{Schilling_SF}, which completes the proof of the Theorem.

\bibliographystyle{plain}

%\bibliography{./bib_pp}
\end{document}